\documentclass[12pt]{amsart}

\usepackage{amsmath,amsthm,amsfonts,amssymb}
\usepackage[margin=1in]{geometry}    
\usepackage{parskip}  

\newtheorem{definition}{Definition}[section]
\newtheorem{theorem}{Theorem}[section]
\newtheorem{lemma}{Lemma}[section]

\newtheorem{prop}{Proposition}[section]
\newtheorem{remark}{Remark}[section]
\newtheorem{hypothesis}{Hypothesis}[section]
\numberwithin{equation}{section}

\newcommand{\R}{\mathbb{R}}
\newcommand{\Z}{\mathbb{Z}}

\newcommand{\N}{\mathbb{N}}

\newcommand{\E}{\mathbb{E}}

\makeatletter
\@namedef{subjclassname@2020}{%
  \textup{2020} Mathematics Subject Classification}
\makeatother

\begin{document}

\title[Small Ball Probabilities for SPDE on Compact Manifolds]{Small Ball Probabilities for the Stochastic Heat Equation on Compact Manifolds}
\author{Jiaming Chen}
\address{Dept. of Mathematics
\\Imperial College
\\London, UK}
\email{j.chen1@imperial.ac.uk}
\begin{abstract}
We consider the stochastic heat equation on a compact smooth Riemannian manifold without boundary satisfying
\begin{equation*}
\partial_tu(t,x)=\frac{1}{2}\Delta_Mu(t,x)+\sigma(t,x,u)\dot{W}(t,x),\quad (t,x)\in\mathbb{R}_+\times M,
\end{equation*}
where $\dot{W}$ is a centered Gaussian noise that is white in time and colored in space. Assuming that $\sigma$ is Lipschitz in $u$ and uniformly bounded, we estimate small ball probabilities for the solution $u$ when $u(0,x)\equiv 0$.
\end{abstract}
\maketitle

\section{Introduction}
Stochastic partial differential equations (SPDEs) of parabolic type, and in particular the stochastic heat equation, form a central class of models in probability theory. In the classical Euclidean setting, the stochastic heat equation
\begin{equation}
\partial_t u(t,x)=\frac{1}{2}\Delta u(t,x)+\sigma(t,x,u)\dot{W}(t,x)
\end{equation}
has been studied extensively using stochastic convolution techniques and the theory of stochastic evolution equations in Hilbert spaces (see \cite{walsh1986anintroductiontostochastic} and \cite{da2014stochastic} for details). These foundational works establish existence, uniqueness, and regularity of solutions under various assumptions on the nonlinearity and the driving noise. More recently, fine properties of solutions, including intermittency, sample path regularity, and tail behavior, have attracted significant attention (see \cite{khoshnevisan2025points, guo2025sample, chen2025global, chen2025parabolic, khoshnevisan2024small, hu2025spatio, qian2026temporal}).

In many applications, however, the underlying state space naturally carries a nontrivial geometric structure. This motivates the study of stochastic heat equations on Riemannian manifolds. Let $M$ be a $d$-dimensional compact smooth Riemannian manifold without boundary. We consider the stochastic heat equation
\begin{equation}
\label{she}
\partial_t u(t,x)=\frac{1}{2}\Delta_M u(t,x)+\sigma(t,x,u)\dot{W}(t,x),\quad (t,x)\in \R_+\times M,
\end{equation}
with deterministic initial condition $u(0,\cdot)=u_0:M\to\R$, where $\Delta_M=\mathrm{div}\circ\mathrm{grad}$ denotes the Laplace-Beltrami operator. A random field $\{u(t,x)\}_{(t,x)\in\R_+\times M}$ is said to be a mild solution to \eqref{she} if it satisfies
\begin{equation}
\label{sol}
u(t,x)=\int_M P_t(x,y)u_0(y)dy+\int_{[0,t]\times M}P_{t-s}(x,y)\sigma(s,y,u(s,y))W(dsdy),
\end{equation}
where $P_t(x,y)$ denotes the heat kernel on $M$. Precise estimates on $P_t$ can be obtained under geometric assumptions such as lower bounds on the Ricci curvature (see \cite{10.1007/BF02399203} and \cite{davies1989pointwise}). These analytic properties play a crucial role in the study of SPDEs on manifolds. The stochastic integral in \eqref{sol} is understood in the It\^o-Walsh sense.

The It\^o-Walsh solution theory for \eqref{she} is well developed in spatial dimension $d=1$ when $\dot{W}$ is space-time white noise. In higher dimensions, however, white noise becomes too singular, and it is well known that no function-valued solution exists for the stochastic heat equation driven by space-time white noise when $d\geq2$.

To study the equation in higher dimensions, it is therefore natural to regularize the noise in the spatial variable. In the Euclidean setting $\R^d$, Dalang \cite{dalang1999extending} extended Walsh’s theory to $d\geq2$ by considering noises that are white in time and colored in space, with homogeneous spatial covariance $G(x,x')=G(x-x')$. A necessary and sufficient condition for the existence of a random field solution is
\begin{equation*}
    \int_{\R^d}\frac{\hat{G}(d\xi)}{1+|\xi|^2}<\infty,
\end{equation*}
where $\hat{G}$ denotes the Fourier transform of $G$. This condition, commonly referred to as Dalang’s condition, can be interpreted as a regularity requirement on the noise.

In contrast to the Euclidean case, where Fourier analysis provides a natural tool via Bochner’s theorem, the construction on a compact manifold relies on the spectral decomposition of the Laplace-Beltrami operator, which plays the role of a Fourier transform in this geometric setting. The statement below is to construct an intrinsic family of Gaussian noises on compact Riemannian manifolds that are smoother than white noise and allow the study of \eqref{she} within the It\^o-Walsh framework. 

Let $L^2(M)$ denote the space of square-integrable functions on $M$, equipped with the inner product $\langle\cdot,\cdot\rangle$. Let $\{\phi_n\}_{n\geq0}$ be an orthonormal basis of $L^2(M)$ consisting of eigenfunctions of $-\Delta_M$, with corresponding eigenvalues $0=\lambda_0<\lambda_1\leq\lambda_2\leq\cdots$. Thus, $-\Delta_M\phi_n=\lambda_n\phi_n$, and for any $\varphi\in L^2(M)$ there exists a unique expansion
\begin{equation}
\label{decomp}
\varphi(x)=\sum_{n\geq 0} a_n\phi_n(x).
\end{equation}
In particular, $a_0=m_0^{-1/2}\int_M\varphi dm$, where $m_0$ denotes the volume of $M$.

To define a function-valued solution to \eqref{she}, we introduce a family of spatial Gaussian noises with covariance structure depending on parameters $\alpha,\rho\geq0$. Let $(\Omega, \mathcal{F}, P)$ be a complete probability space and let $W(\varphi)$ and $W(\psi)$ be centered Gaussian random variables for each $\varphi$ and $\psi$ in $L^2(M)$, with covariance
\begin{equation}
\label{cov}
\E[W(\varphi)W(\psi)] = \langle\varphi, \psi\rangle_{\alpha,\rho}:=\rho a_0b_0 + \sum_{n\neq 0} \frac{a_nb_n}{\lambda_n^\alpha},
\end{equation}
where $\{a_n\}$ and $\{b_n\}$ are the coefficients of $\varphi$ and $\psi$ in decomposition \eqref{decomp}, respectively. For $\rho>0$, let $\mathcal{H}^{\alpha,\rho}$ be the completion of $L^2(M)$ with respect to $\langle\cdot,\cdot\rangle_{\alpha,\rho}$.

\begin{remark}
From \eqref{cov}, we have the continuous embeddings $L^2(M)\subset\mathcal{H}^{\alpha,\rho}\subset\mathcal{H}^{\beta,\rho}$ for $0\leq\alpha<\beta$. Moreover, the choice $\rho=1$ and $\alpha=0$ recovers spatial white noise on $M$.
\end{remark}

For $\alpha,\rho>0$, the covariance $\langle\cdot,\cdot\rangle_{\alpha,\rho}$ admits an integral kernel representation. Define
\begin{equation*}
G_{\alpha}(x,y):=\frac{1}{\Gamma(\alpha)}\int_0^\infty t^{\alpha-1} \left(P_t(x,y)-\frac{1}{m_0}\right)dt\text{ and }G_{\alpha,\rho}(x,y):=\frac{\rho}{m_0} + G_{\alpha}(x,y).
\end{equation*}
Using the spectral expansion
\begin{equation*}
    P_t(x,y) = \frac{1}{m_0}+\sum_{n=1}^\infty e^{-\lambda_nt}\phi_n(x)\phi_n(y)
\end{equation*}
and the identity $\int_0^\infty t^{\alpha-1}e^{-\lambda t}dt=\Gamma(\alpha)\lambda^{-\alpha}$, we obtain
\begin{equation*}
    G_{\alpha,\rho}(x,y)=\frac{\rho}{m_0}+\sum_{n=1}^\infty\lambda_n^{-\alpha}\phi_n(x)\phi_n(y).
\end{equation*}
Consequently,
\begin{equation}
\label{covker}
\langle\varphi,\psi\rangle_{\alpha,\rho}=\int_{M^2}\varphi(x)G_{\alpha,\rho}(x,y)\psi(y)m(dx)m(dy).
\end{equation}

\begin{remark}
    Throughout the paper, we assume that Dalang’s condition $\alpha>(d-2)/2$ holds.
\end{remark}

\begin{remark}
The kernel $G_\alpha$ can be viewed as an analogue of the Riesz kernel on $\R^d$. Although $\int_M G_\alpha(x,y),m(dy)=0$ and $G_\alpha$ is not nonnegative, it was shown in \cite{chen2025parabolic} that $G_\alpha$ is bounded below on $M$. Hence, by choosing $\rho$ sufficiently large, the covariance function $G_{\alpha,\rho}$ can be made nonnegative.
\end{remark}

\begin{prop}[Lemma 2.9 in \cite{brosamler1983laws}]
\label{covker_bound}
For any $\alpha>0$, there exists a constant $C_\alpha>0$ such that
\begin{equation*}
\vert G_{\alpha}(x,y)\vert\leq\begin{cases}
    C_{\alpha} &\alpha>d/2\\
    C_{\alpha}(1+\log^-(d(x,y))) &\alpha=d/2\\
    C_{\alpha} d(x,y)^{2\alpha-d} &\alpha<d/2
\end{cases}
\end{equation*}
where $\log^-(z)=\max(0,-\log z)$ and $d(x,y)$ denotes the Riemannian distance on $M$.
\end{prop}

We now define the noise on $\R_+\times M$ that is white in time and colored in space.
\begin{definition}
\label{color_noise}
Let $\alpha > 0$ and define the following Hilbert space of space-time functions
\[
\mathcal{H}_{\alpha,\rho} := L^2(\R_+;\mathcal{H^{\alpha,\rho}}(M))
\]
equipped with $\langle\cdot,\cdot\rangle_{\mathcal{H}_{\alpha,\rho}}$ such that for any $\varphi,\psi\in\mathcal{H}_{\alpha,\rho}$,
\[
\langle\varphi,\psi\rangle_{\mathcal{H}_{\alpha,\rho}}:=\int_{0}^\infty\langle\varphi(t,\cdot), \psi(t,\cdot)\rangle_{\alpha,\rho}dt.
\]
On a complete probability space $(\Omega, \mathcal{F},P)$, let $W_{\alpha,\rho}$ be a centered isonormal Gaussian process over $\mathcal{H}_{\alpha,\rho}$, i.e.,
\[
\E[W_{\alpha,\rho}(\varphi)W_{\alpha,\rho}(\psi)]=\langle\varphi,\psi\rangle_{\mathcal{H}_{\alpha,\rho}}
\]
for each $\varphi,\psi$ in $\mathcal{H}_{\alpha,\rho}$. We refer to $W_{\alpha,\rho}$ as colored noise on $M$ that is white in time.
\end{definition}

We impose the following assumptions on the noise coefficient $\sigma$.
\begin{hypothesis}
There exists a constant $\mathcal{D}>0$ such that for all $t\geq 0$, $x\in M$ and $u,v\in\R$,
\begin{equation}
\label{hypothesis1}
\vert \sigma(t,x,u)-\sigma(t,x,v)\vert\leq\mathcal{D}|u-v|.
\end{equation}
\end{hypothesis}
\begin{hypothesis}
There exist constants $\mathcal{C}_1$, $\mathcal{C}_2>0$ such that for all $t\geq 0$, $x\in M$ and $u\in\R$,
\begin{equation}
\label{hypothesis2}
\mathcal{C}_1\leq \sigma(t,x,u)\leq \mathcal{C}_2.
\end{equation}
\end{hypothesis}

\begin{remark}
    The first assumption ensures that $\sigma$ is Lipschitz continuous in the solution variable, while the second guarantees uniform nondegeneracy of the noise. Under Hypothesis \eqref{hypothesis1} and Dalang’s condition, it is well known that equation \eqref{she} admits a unique random field solution.
\end{remark}

Small ball problems have been well studied and one can see \cite{li2001gaussian} for an overview of known results on Gaussian processes and references on other processes. In summary, given a stochastic process $X_t$ starting at the origin, the small ball problem concerns the asymptotic behavior of
\[P\left(\sup_{0\leq t\leq T}\vert X_t\vert<\varepsilon\right)
\text{ as } \varepsilon\downarrow 0.\]

In our framework, the solution $u$ is generally non-Gaussian due to the dependence of $\sigma$ on $u$. Small ball probabilities for stochastic heat equations of this type were first studied in \cite{athreya2021small} for space-time white noise. Subsequent extensions include H\"older-type norms \cite{foondun2023small} and diffusion coefficients that are uniformly elliptic but only H\"older continuous in $u$ \cite{han2024support}. These works, however, are restricted to one-dimensional spatial domains due to the singularity of white noise.

More recently, \cite{chen2025small} adapted the approach of \cite{athreya2021small}, combined with Dalang’s framework \cite{dalang1999extending}, to study small ball probabilities for the fractional stochastic heat equation driven by spatially homogeneous noise on the $d$-dimensional torus. Their analysis relies on Fourier methods to control fluctuations of the noise term and employs a different temporal discretization scheme. 

The present work extends the study of small ball probabilities for stochastic heat equations to arbitrary compact smooth Riemannian manifolds. Building on the It\^o-Walsh framework and Dalang’s theory of spatially colored noise, we replace Euclidean Fourier analysis with the spectral decomposition of the Laplace–Beltrami operator, yielding an intrinsic and coordinate-free approach that explicitly reflects the underlying geometry. Within this setting and under the conditions on the noise coefficient, we establish small ball probability estimates for the resulting non-Gaussian solutions, thereby extending the one-dimensional theory developed in \cite{athreya2021small, foondun2023small, han2024support} and the torus-based analysis of \cite{chen2025small} to a broad class of compact manifolds. In particular, our approach allows us to treat the boundary regime $\alpha=d/2$, which is inaccessible in previous works relying on spatial homogeneity or flat geometry. This work can be viewed as an initial step toward a study of small ball probabilities for SPDEs on manifolds.

\section{Main Result}
\begin{theorem}
\label{thm} 
Under the hypothesis \eqref{hypothesis1} and \eqref{hypothesis2} on the noise coefficient $\sigma$, consider the solution to \eqref{she} with initial condition $u_0\equiv 0$. Then there exist positive constants $\textbf{C}_0,\textbf{C}_1,\textbf{C}_2,\textbf{C}_3,\mathcal{D}_0$ and $\varepsilon_0$, independent of $\varepsilon$, such that for all $0<\mathcal{D}<\mathcal{D}_0,0<\varepsilon<\varepsilon_0$ and $T>1$,
                \begin{equation*}
\begin{aligned}
\mathbf{C}_0 \exp\left(-\frac{\mathbf{C}_1 T}{\varepsilon^{\frac{2d+4}{h}}}\right)
&< P\left(\sup_{\substack{0<t\le T \\ x\in M}} |u(t,x)| < \varepsilon\right) \\
&<
\begin{cases}
\mathbf{C}_2 \exp\left(-\dfrac{\mathbf{C}_3 T}{\varepsilon^{\frac{H}{2\alpha+2-d}}}\right)
& \max\left(0,\dfrac{d}{2}-1\right)<\alpha<\dfrac{d}{2} \\[0.8em]
\mathbf{C}_2 \exp\left(-\dfrac{\mathbf{C}_3 T [\ln\ln(1/\varepsilon)]^2}{\varepsilon^{2}}\right)
& \alpha=\dfrac{d}{2} \\[0.8em]
1 & \alpha>\dfrac{d}{2}.
\end{cases}
\end{aligned}
\end{equation*}
        Here $h=\min(1,2\alpha-d+2)$ and $H=\min\left(\frac{6d-4\alpha}{d},4\alpha+8-2d\right)$.
\end{theorem}

\begin{remark}
Except in the one-dimensional white noise case $(d=1,\alpha=0)$, the exponent of $\varepsilon$ appearing in the upper bound is strictly smaller than that in the lower bound. This discrepancy arises from the fact that the temporal coefficient $c_0$, defined in \eqref{c0}, depends explicitly on $\varepsilon$. Despite this suboptimality, the small ball estimates remain valid in all dimensions for both the upper and lower bounds. 
\end{remark}

\begin{remark}
    When $\alpha>d/2$, the driving noise is bounded and continuous, and its spatial correlations decay to a constant as the distance between two points increases. In this regime, the corresponding small ball probability admits only a trivial upper bound within our framework. 
\end{remark}

We now outline the main ideas underlying the proof of Theorem \ref{thm}. In \cite{athreya2021small}, the authors decomposed a one-dimensional spatial interval into small subintervals. To extend this approach to a $d$-dimensional compact manifold, we instead partition the domain into a collection of nested geodesic balls.

A key technical difficulty concerns the temporal coefficient $c_0$ introduced in \cite{athreya2021small}. In that work, $c_0$ was shown to be uniformly bounded, relying on the exponential decay of spatial correlations as the distance between points increases. In our setting, however, the spatial correlations decay much more slowly, and a uniform bound on $c_0$ is no longer available. Instead, we control $c_0$ in terms of $\varepsilon$, as specified in \eqref{c0}.

A central ingredient of the analysis is the Markov property of \eqref{she} with respect to the time variable $t$. This allows us to reduce the problem to studying the behavior of the (non-Gaussian) solution over short time intervals. To derive the upper bound in Theorem \ref{thm}, we employ a perturbative argument that approximates the solution $u$ by a Gaussian random field in small space-time regions. This approximation is combined with sharp estimates for an associated Gaussian process, whose behavior depends sensitively on the noise parameter $\alpha$. This dependence explains the specific choice of upper bounds in the sequence of events appearing in the next section.

For the lower bound, we apply the Gaussian correlation inequality together with a change-of-measure argument inspired by \cite{athreya2021small}. We further show that the approximation error can be controlled by an appropriate choice of the time intervals over which the coefficients are frozen. A carefully constructed stopping time plays a crucial role here, allowing us to overcome the obstruction that prevents a direct extension of the lower bound argument to dimensions $d>1$.

In particular, it is natural to ask how the present estimates behave as the Riesz kernel parameter approaches its critical value $\alpha=d/2$. This regime is closely related to the weak convergence problem as $\beta\downarrow 0$ studied in \cite{bezdek2016weak}, as well as Remark(a) following Theorem 2.1 in \cite{chen2024small}. Our results complement and extend the Euclidean theory of SPDEs and demonstrate that the entire framework can be naturally generalized to compact Riemannian manifolds, thereby covering a broad class of spaces with rich geometric and topological structure.

The remainder of the paper is organized as follows. In Section \ref{keyprop}, we state Proposition \ref{prop} and explain its relationship to Theorem \ref{thm}. Section \ref{estimates} collects several auxiliary estimates, and Section \ref{proofprop} contains the proof of Proposition \ref{prop}, thereby completing the argument.

We conclude by listing some conventions and notations used throughout the paper.
\begin{enumerate}
\item The symbols $C$ and $C'$ denote positive constants whose values may change from line to line. Dependence on specific parameters is indicated explicitly when relevant.
\item For notational simplicity, we suppress the indices $\alpha$ and $\rho$ and write $W$ in place of $W_{\alpha,\rho}$. We also write $dz$ instead of $m(dz)$ for integration over $M$.
\item $B_M(p,r)$ denotes the geodesic ball in $M$ centered at $p$ with radius $r$.
\item $\mathrm{diam}(M)$ denotes the diameter of the compact manifold $M$.
\end{enumerate}

\section{Key Proposition}
\label{keyprop}
Fix $\varepsilon>0$ and choose a maximal $\varepsilon^2$-separated set $\{x_j\}_{j=1}^{N(\varepsilon)}\subset M$,
that is, $d(x_j,x_k)\geq\varepsilon^2$ for $j\neq k$ and $\bigcup\limits_{j=1}^{N(\varepsilon)}B_M(x_j,\varepsilon^2)=M$. Next, we partition the time interval $[0,T]$ into subintervals of length $c_0\varepsilon^4$ where $c_0=c_0(\varepsilon)$ is a temporal coefficient satisfying
\begin{equation}
\label{c0}
c_0<\begin{cases}
    \exp\left(C\varepsilon^{-2d}\ln(\varepsilon)\right)&\alpha=d/2\\
    C\varepsilon^{\frac{8\alpha}{d-2\alpha}}&\alpha<d/2.
\end{cases}
\end{equation}
The constant $C>0$ defined in \eqref{selectc0} is independent of $\varepsilon$.

\begin{remark}
Unlike the space-time white noise case studied in \cite{athreya2021small}, the coefficient $c_0$ must depend on $\varepsilon$ due to the strong spatial correlations between distant points. As shown in the Lemma \ref{coeffbound}, when $\alpha<d/2$ the correlation sum decays polynomially; when $\alpha=d/2$, it decays logarithmically; and when $\alpha>d/2$ it remains bounded away from zero, in which case the correlation sum cannot be controlled for any choice of $c_0$ as $\varepsilon\downarrow0$. Nevertheless, in our setting, the precise value of $c_0$ does not enter either the upper or lower bounds for the small ball probabilities.
\end{remark}

We denote $t_i=ic_0\varepsilon^4$ and pick a reference point $x_0\in M$ to define the nested sets for $j=1,2,\dots,J(\varepsilon)$,
\begin{equation}
\label{Rij}
R_{i,j}:=\left\lbrace(t_i,x_k):d(x_0,x_k)\leq j\varepsilon^2
    \right\rbrace
\end{equation}
where $J(\varepsilon)\geq\frac{\text{diam}(M)}{\varepsilon^2}>J(\varepsilon)-1$. Thus, $R_{i,j}$ consists of all spatial separation points contained in the geodesic ball $B_M(x_0,j\varepsilon^2)$ at time $t_i$. 

According to Proposition~\ref{covker_bound}, the solution to \eqref{she} exhibits distinct behaviors depending on the value of $\alpha$. We therefore introduce the following function, which will be used throughout the paper:
\begin{equation}
\label{f_t}
        f(t) = \begin{cases}
            t & \alpha>d/2\\
            t\ln(1/t) & \alpha=d/2\\
            t^{\alpha+1-d/2} & \max(0,d/2-1)<\alpha<d/2.
        \end{cases}
    \end{equation}

For $n\geq0$, we define the events
\begin{equation}
\label{Fn}
F_n=\left\lbrace\vert u(t,x)\vert\leq \sqrt{f(t_1)}\text{ for all } (t,x)\in R_{n,J}\right\rbrace,
\end{equation}
which capture the boundedness of the solution at time $t_n$ over all spatial separation points.

In addition, let $E_{-1}=\Omega$ and recall $h=\min(1,2\alpha-d+2)$. For $n\geq0$, we define
\begin{equation}
\label{En}
E_n=\left\lbrace\vert u(t_{n+1},x)\vert\leq \frac{\mathcal{C}_3}{3}t_1^{h/4},\,\text{and }\vert u(t,x)\vert\leq \mathcal{C}_3t_1^{h/4}\text{ for all } t\in[t_n,t_{n+1}),x\in M\right\rbrace.
\end{equation}
This event enforces a uniform bound on the solution over the entire time interval $[t_n,t_{n+1})$, with a stricter constraint at time $t_{n+1}$. The constant $\mathcal{C}_3>0$ is chosen so that
\begin{equation}
\label{c3}
\mathcal{C}_3>6\mathcal{C}_2\sqrt{\frac{\ln C_5}{C_6}},
\end{equation}
where $\mathcal{C}_2$ is the uniform bound of $\sigma$ in \eqref{hypothesis2} and $C_5,C_6$ are positive constants appearing in Lemma \ref{larged}.

\begin{prop}
\label{prop}
Consider the solution to \eqref{she} with the initial condition $u_0\equiv 0$. There exist positive constants $\textbf{C}_4,\textbf{C}_5,\textbf{C}_6,\textbf{C}_7,\mathcal{D}_0$ and $\varepsilon_1$, independent of $\varepsilon$, such that for all $0<\varepsilon<\varepsilon_1$ and $0<\mathcal{D}<\mathcal{D}_0$ in \eqref{hypothesis1}, 
\begin{enumerate}
\item[(a)]
\begin{equation*}
P\left(F_n\bigg| \bigcap_{k=0}^{n-1}F_k\right)\leq\begin{cases}
    \textbf{C}_4\exp\left(-\frac{\textbf{C}_5}{\varepsilon^2}\right)&(d-1)/2<\alpha\leq d/2\\
    \textbf{C}_4\exp\left(-\frac{\textbf{C}_5}{t_1^{\min\left(\frac{d-2\alpha}{2d},\frac{2\alpha+2-d}{2}\right)}}\right)&\max(0,d/2-1)<\alpha\leq(d-1)/2,
\end{cases}
\end{equation*}
\item[(b)] 
\begin{equation*}
P\left(E_n\bigg| \bigcap_{k=-1}^{n-1}E_k\right)\geq \textbf{C}_6\exp\left(-\frac{\textbf{C}_7}{t_1^{d/2}}\right).
\end{equation*}
\end{enumerate}
\end{prop}
We now demonstrate how Theorem \ref{thm} follows from Proposition \ref{prop}.

\textbf{Proof of Theorem \ref{thm}}
\begin{proof}
The event $F_{n}$ in $\eqref{Fn}$ controls the behavior of $u(t,x)$ at the discrete time $t_n$. Taking the intersection over all such times yields
\[F:=\bigcap_{n=0}^{\left\lfloor\frac{T}{t_1}\right\rfloor}F_{n}\supset \left\lbrace\vert u(t,x)\vert\leq \sqrt{f(t_1)}, t\in[0,T],x\in M\right\rbrace.\]
Moreover,
\[P(F)=P\left(\bigcap_{n=0}^{\left\lfloor\frac{T}{t_1}\right\rfloor}F_{n}\right)=P(F_{0})\prod_{n=1}^{\left\lfloor\frac{T}{t_1}\right\rfloor}P\left(F_{n}\bigg| \bigcap_{k=0}^{n-1}F_{k}\right).\]
Since $u_0(x)\equiv 0$, $F_0=\Omega$, and Proposition \ref{prop}(a) therefore implies that, for $\max(0,d/2-1)<\alpha\leq(d-1)/2$,
\begin{align*}
P\left(\left\lbrace\vert u(t,x)\vert\leq t_1^\frac{2\alpha+2-d}{4}, t\in[0,T],x\in M\right\rbrace\right)&\leq P(F)\leq \left[\textbf{C}_4\exp\left(-\frac{\textbf{C}_5}{t_1^{\min\left(\frac{d-2\alpha}{2d},\frac{2\alpha+2-d}{2}\right)}}\right)\right]^{\left\lfloor\frac{T}{t_1}\right\rfloor}\\
&\leq \textbf{C}_2\exp\left(-\frac{\textbf{C}_3T}{t_1^{\min\left(\frac{3d-2\alpha}{2d},\frac{2\alpha+4-d}{2}\right)}}\right).
\end{align*}
Replacing $t_1^\frac{2\alpha+2-d}{4}$ with $\varepsilon_2$ yields
\begin{equation*}
P\left(\left\lbrace\vert u(t,x)\vert\leq \varepsilon_2, t\in[0,T],x\in M\right\rbrace\right)\leq\textbf{C}_2\exp\left(-\frac{\textbf{C}_3T}{\varepsilon_2^{\min\left(\frac{6d-4\alpha}{d(2\alpha+2-d)},\frac{4\alpha+8-2d}{2\alpha+2-d}\right)}}\right).
\end{equation*}

For $(d-1)/2<\alpha< d/2$, Proposition \ref{prop}(a) gives
\[
P\left(\left\lbrace\vert u(t,x)\vert\leq t_1^\frac{2\alpha+2-d}{4}, t\in[0,T],x\in M\right\rbrace\right)\leq\textbf{C}_2\exp\left(-\frac{\textbf{C}_3T}{t_1^{\frac{3d-2\alpha}{2d}}}\right),
\]
and replacing $t_1^\frac{2\alpha+2-d}{4}$ with $\varepsilon_2$ yields
\begin{equation*}
P\left(\left\lbrace\vert u(t,x)\vert\leq \varepsilon_2, t\in[0,T],x\in M\right\rbrace\right)\leq\textbf{C}_2\exp\left(-\frac{\textbf{C}_3T}{\varepsilon_2^{\frac{6d-4\alpha}{d(2\alpha+2-d)}}}\right).
\end{equation*}

Note that the exponent $\frac{6d-4\alpha}{d(2\alpha+2-d)}$ is decreasing on $(d-1)/2<\alpha< d/2$, and
\begin{equation*}
    \inf_{\alpha}\frac{6d-4\alpha}{d(2\alpha+2-d)}=2,
\end{equation*}
as $\alpha\to \frac{d}{2}^-$. At the critical value $\alpha=d/2$, an additional logarithmic correction appears. Indeed,
\begin{equation*}
    P\left(\left\lbrace\vert u(t,x)\vert\leq \sqrt{t_1\ln(1/t_1)}, t\in[0,T],x\in M\right\rbrace\right)\leq\textbf{C}_2\exp\left(-\frac{\textbf{C}_3T}{t_1\varepsilon^2}\right).
\end{equation*}
Replacing $\sqrt{t_1\ln(1/t_1)}$ with $\varepsilon_2$, we obtain
\begin{equation*}
    t_1\varepsilon^2\leq\frac{C\varepsilon_2^2\varepsilon^{2+2d}}{\ln(1/\varepsilon)}\leq\frac{C\varepsilon_2^2\ln(1/\varepsilon_2)^{-C'}}{\ln\ln(1/\varepsilon_2)}\leq\frac{C\varepsilon_2^2}{[\ln\ln(1/\varepsilon_2)]^{2}},
\end{equation*}
which implies
\begin{equation*}
    P\left(\left\lbrace\vert u(t,x)\vert\leq \varepsilon_2, t\in[0,T],x\in M\right\rbrace\right)\leq\textbf{C}_2\exp\left(-\frac{\textbf{C}_3T[\ln\ln(1/\varepsilon_2)]^2}{\varepsilon_2^{2}}\right).
\end{equation*}

We now turn to the lower bound. The event $E_{n}$ in $\eqref{En}$ deals with the behavior of $u(t,x)$ over the entire interval $[t_n,t_{n+1}]$. Hence
\[E:=\bigcap_{n=-1}^{\left\lfloor\frac{T}{t_1}\right\rfloor}E_{n}\subset \left\lbrace\vert u(t,x)\vert\leq t_1^{h/4}, t\in[0,T],x\in M\right\rbrace,\]
and
\[P(E)=P\left(\bigcap_{n=-1}^{\left\lfloor\frac{T}{t_1}\right\rfloor}E_{n}\right)=P(E_{-1})\prod_{n=0}^{\left\lfloor\frac{T}{t_1}\right\rfloor}P\left(E_{n}\bigg| \bigcap_{k=-1}^{n-1}E_{k}\right).\]
With $u_0(x)\equiv 0$ and $E_{-1}=\Omega$, Proposition \ref{prop}(b) immediately yields
\begin{equation*}
    \begin{split}
        P\left(\left\lbrace\vert u(t,x)\vert\leq \mathcal{C}_3t_1^{h/4}, t\in[0,T],x\in M\right\rbrace\right)&\geq P(E)\geq \left[\textbf{C}_6\exp\left(-\frac{\textbf{C}_7}{t_1^{d/2}}\right)\right]^{\left\lfloor\frac{T}{t_1}\right\rfloor+1}\\
        &\geq \textbf{C}_0\exp\left(-\frac{\textbf{C}_1T}{t_1^{d/2+1}}\right).
    \end{split}
\end{equation*}
Finally, replacing $t_1^{h/4}$ with $\varepsilon_3$ gives
\begin{equation*}
P\left(\left\lbrace\vert u(t,x)\vert\leq \varepsilon_3, t\in[0,T],x\in M\right\rbrace\right)\geq\textbf{C}_0\exp\left(-\frac{\textbf{C}_1T}{\varepsilon_3^{\frac{2d+4}{h}}}\right).
\end{equation*}
The proof is complete, and the rest of this paper is devoted to the proof of Proposition \ref{prop}.
\end{proof}

\section{Preliminaries}
\label{estimates} 
In this section, we recall basic facts on the heat kernel on a compact Riemannian manifold and collect several estimates that will be used throughout the paper.

\subsection{Heat kernel estimates}
Let $M$ be a compact $d$-dimensional Riemannian manifold without boundary. A function $P:(0,\infty)\times M \times M\to\R$ is called a fundamental solution (or heat kernel) of \eqref{she} if \begin{enumerate}
    \item $P\in C^{1,2,2}\left((0,\infty)\times M\times M\right)$;
    \item for all $x,y\in M$ and $t>0$,
    \[
        \partial_t P(t,x,y) = \frac{1}{2}\Delta_x P(t,x,y);
    \]
    \item for every $x\in M$ and $f\in C(M)$,
    \[
        \lim_{t\downarrow 0}\int_M P(t,x,y)f(y)\,dy = f(x).
    \]
\end{enumerate}
We write $P_t(x,y)=P(t,x,y)$ to simplify our notation. A fundamental solution $P$ can be constructed in \cite{Minakshisundaram:1953xh} with the method of parametrix (see also \cite{minakshisundaram1949some}). We recall the following estimate for small $t$, in \cite{Minakshisundaram:1953xh}, for all $n\geq0$, there exists $C>0$ such that
\begin{equation}
\label{heatkernelbound1}
    P_t(x,y)\leq (2\pi t)^{-\frac{d}{2}}\exp\left(-\frac{d(x,y)^2}{2t}\right)+Ct^n\quad x,y\in M,t\leq 1,
\end{equation}
and the following estimate for large $t\geq 1$ from \cite{baxter1976energy}: There exist $C,C'>0$ such that
\begin{equation}
    \label{heatkernelbound2}
    \sup_{x,y\in M}\vert P_t(x,y)-m_0^{-1}\vert\leq C\exp(-C't).
\end{equation}

The next lemma provides $L^1$-bounds on spatial and temporal increments of the heat kernel. These estimates extend Lemma 4.1 in \cite{chen2025small}.
\begin{lemma}
\label{heat_kernel_integral_bounds}
    There exist positive constants $C,C'$ depending only on $M$ such that, for all $x,y\in M$ and all $0\leq s\leq t\leq 1$,
    \begin{equation}
    \label{heat_kernel_spatial_integral_bound}
        \int_M\vert P_t(x,w)-P_t(y,w)\vert dw\leq \min\left(\frac{Cd(x,y)}{\sqrt{t}},2\right),
    \end{equation}
    \begin{equation}
    \label{heat_kernel_temporal_integral_bound}
        \int_M\vert P_t(x,w)-P_s(x,w)\vert dw\leq \min\left(C'\log\left(\frac{t}{s}\right),2\right).
    \end{equation}
\end{lemma}
\begin{proof}
    Let $\gamma:[0,1]\to M$ be a minimizing geodesic jointing $x$ and $y$ with $\gamma(0)=x$ and $\gamma(1)=y$. For each fixed $w\in M$, the map $s\mapsto P_t(\gamma(s),w)$ is differentiable. By the mean value theorem,
    \begin{equation}
    \label{mean_value_theorem}
        \vert P_t(x,w)-P_t(y,w)\vert\leq d(x,y)\sup_{z\in \text{Im}(\gamma)}\vert\nabla_xP_t(z,w)\vert.
    \end{equation}
For a compact $n$-dimensional Riemannian manifold with Ricci curvature bounded below, \cite{10.1007/BF02399203} concluded that the heat kernel satisfies Gaussian-type gradient estimates for small $t$:
\begin{equation}
\label{heat_kernel_gradient}
    \vert\nabla_xP_t(x,w)\vert\leq \frac{C}{\sqrt{t}V(x,\sqrt{t})}\exp\left(-\frac{C'd(x,w)^2}{t}\right),
\end{equation}
where $V(x,r)$ denotes the volume of the geodesic ball of radius $r$ centered at $x$. Integrating \eqref{mean_value_theorem} over $w \in M$ and applying \eqref{heat_kernel_gradient} gives
\begin{equation*}
\begin{split}
        \int_M\vert P_t(x,w)-P_t(y,w)\vert dw&\leq d(x,y)\sup_{z\in \text{Im}(\gamma)}\int_M\vert\nabla_xP_t(z,w)\vert dw\\
    &\leq d(x,y)\sup_{z\in \text{Im}(\gamma)}\int_M\frac{C}{\sqrt{t}V(z,\sqrt{t})}\exp\left(-\frac{C'd(z,w)^2}{t}\right) dw.
\end{split}
\end{equation*}
Via the Bishop-Gromov volume comparison theorem, $M$ satisfies a local volume doubling property for small $t$, which implies
\begin{equation*}
    \int_M\exp\left(-\frac{C'd(z,w)^2}{t}\right) dw\leq CV(z,\sqrt{t}).
\end{equation*}
Hence,
\begin{equation*}
    \int_M\vert P_t(x,w)-P_t(y,w)\vert dw\leq \frac{Cd(x,y)}{\sqrt{t}}.
\end{equation*}
Since the integral is trivially bounded by $2$, \eqref{heat_kernel_spatial_integral_bound} follows.

For the temporal bound, since $P_t(x,w)$ is $C^1$ in $t$,
\begin{equation*}
\begin{split}
        \int_M\vert P_t(x,w)-P_s(x,w)\vert dw &\leq \int_M\int_s^t\vert \partial_rP_r(x,w)\vert drdw\\
        &\leq \int_{s}^{t}\int_M\frac{C}{rV(x,\sqrt{r})}\exp\left(-\frac{C'd(x,w)^2}{r}\right)dwdr\\
        &\leq C\int_s^t\frac{1}{r}dr= C\log\left(\frac{t}{s}\right),
\end{split}
\end{equation*}
where the second inequality follows from equation $(4.1)$ in \cite{coulhon2020gradient} and Fubini's theorem. Again, the integral is clearly bounded by $2$, giving \eqref{heat_kernel_temporal_integral_bound}.
\end{proof}

\subsection{Noise term estimates}
We denote the stochastic integral term in \eqref{sol} (the noise term) by
\begin{equation}
    \label{noise_term}
    N(t,x):=\int_{[0,t]\times M}P_{t-s}(x,y)\sigma(s,y,u(s,y))W(dsdy),
\end{equation}
and we will estimate its regularity in the following lemmas.

\begin{lemma}[Spatial regularity of the noise term]\label{spatialregularity} There exists a positive constant $C$, independent of $\mathcal{C}_2$ in \eqref{hypothesis2}, such that for any $\alpha>\max\left(\frac{d-2}{2},0\right)$, $\xi\in(0,\min\left(2\alpha-d+2,1\right))$, $t\in[0,1]$ and $x,y\in M$, we have
\begin{equation*}
\mathbb{E}\left[(N(t,x)-N(t,y))^2\right]\leq\begin{cases}
    C\mathcal{C}_2^2d(x,y)^{2\xi} &\alpha>d/2\\
    C\mathcal{C}_2^2d(x,y)^{\xi} &\alpha\leq d/2.
    \end{cases}
\end{equation*}
\end{lemma}
\begin{proof}
To simplify our notation, we fix $t,x$ and y, and define, for $0<s<t$,
\begin{equation*}
    K_s(w):=P_{t-s}(x,w)-P_{t-s}(y,w).
\end{equation*}
By Definition \ref{color_noise} and \eqref{covker}, we have
\begin{equation}
\label{spatial_regularity_proof}
\begin{split}
&\mathbb{E}\left[(N(t,x)-N(t,y))^2\right]= \mathbb{E}\left[W^2_{\alpha,\rho}(K_t(\cdot)\sigma(t,\cdot,u(t,\cdot)))\right]\\
&\leq \mathbb{E}\int_0^t\int_{M^2}\left\vert K_s(z)K_s(w)G_{\alpha,\rho}(w,z)\sigma(s,z,u(s,z))\sigma(s,w,u(s,w))\right\vert dwdzds\\
&\leq \sup_{r,v}\mathbb{E}\left[\sigma(r,v,u(r,v))^2\right]\int_0^t\int_{M^2}\vert K_s(z)\vert\cdot \vert K_s(w)\vert\cdot G_{\alpha,\rho}(w,z)dwdzds.
\end{split}
\end{equation}

\textbf{Case 1: $\alpha>d/2$}

By \eqref{hypothesis2}, Proposition \ref{covker_bound}, and Lemma \ref{heat_kernel_integral_bounds},
\begin{equation*}
\begin{split}
\mathbb{E}\left[(N(t,x)-N(t,y))^2\right]&\leq C\mathcal{C}_2^2
\int_0^t\int_{M^2}\vert K_s(z)\vert\cdot \vert K_s(w)\vert dwdzds\\
&\leq C\mathcal{C}_2^2\int_0^t\min\left(\frac{Cd(x,y)}{\sqrt{t-s}},2\right)^2ds\\
&\leq C\mathcal{C}_2^2d(x,y)^{2\xi}\int_0^t(t-s)^{-\xi}ds\leq C\mathcal{C}_2^2d(x,y)^{2\xi},
\end{split}
\end{equation*}
where we used $\min(cx,2)\leq\max(c,2)x^\xi$ for $x>0, 0<\xi<1$ and $\int_0^t(t-s)^{-\xi}ds\leq\frac{1}{1-\xi}$.

\textbf{Case 2: $\alpha=d/2$}

Using H\"older's inequality and the integrability of the logarithmic singularity on compact manifolds: for $p \in (1,\infty)$,
\begin{equation*}
        \Vert\log^-(d(\cdot,z))P_t(x,\cdot)\Vert_{L^1}\leq\Vert P_t(x,\cdot)\Vert_{L^p}\cdot\Vert \log^-(d(\cdot,z))\Vert_{L^{\frac{p}{p-1}}}\leq Ct^{-\frac{d}{2}(1-\frac{1}{p})}.
\end{equation*}
Applying Proposition \ref{covker_bound}, we obtain
\begin{equation*}
    \begin{split}
\mathbb{E}&\left[(N(t,x)-N(t,y))^2\right]\leq C\mathcal{C}_2^2
\int_0^t\int_{M^2}\vert K_s(z)\vert\cdot \vert K_s(w)\vert \cdot(1+\log^-(d(w,z))) dwdzds\\
&\leq C\mathcal{C}_2^2\left[d(x,y)^{2\xi}+\int_0^t\int_M\vert K_s(z)\vert \left(\int_M\log^-(d(w,z))[P_{t-s}(x,w)+P_{t-s}(y,w)] dw\right)dzds\right]\\
&\leq C\mathcal{C}_2^2\left(d(x,y)^{2\xi}+d(x,y)^{\xi}\int_0^t(t-s)^{-\frac{\xi}{2}-\frac{d}{2}(1-\frac{1}{p})}ds\right)\leq C\mathcal{C}_2^2d(x,y)^{\xi},
\end{split}
\end{equation*}
where the integral $\int_0^t(t-s)^{-\frac{\xi}{2}-\frac{d}{2}(1-\frac{1}{p})}ds$ converges when choosing $p<\frac{d}{d-2+\xi}$.

\textbf{Case 3: $\alpha<d/2$}
\begin{equation}
\label{heat_kernel_conv}
\begin{split}
        \int_MP_{t-s}(x,w)G_{\alpha,\rho}(w,z)dw&=\int_MP_{t-s}(x,w)\left(\frac{\rho}{m_0}+\frac{1}{\Gamma(\alpha)}\int_0^\infty r^{\alpha-1} \left(P_r(w,z)-\frac{1}{m_0}\right)dr\right)dw\\
        &=C+C'\int_0^\infty r^{\alpha-1}\left(P_{t-s+r}(x,z)-\frac{1}{m_0}\right)dr\\
        &\leq C+C'\int_0^\infty r^{\alpha-1}(t-s+r)^{-d/2}dr\\
        &=C+C'(t-s)^{\alpha-\frac{d}{2}}\int_0^\infty u^{\alpha-1}(1+u)^{-d/2}du\leq C(t-s)^{\alpha-\frac{d}{2}},
\end{split}
\end{equation}
where integral interchanging follows the equation $(2.8)$ in \cite{brosamler1983laws} and $\int_0^\infty u^{\alpha-1}(1+u)^{-d/2}du$ converges when $0<\alpha<d/2$. Similarly, the same estimate holds for $P_{t-s}(y,w)$. Substituting \eqref{heat_kernel_conv} into \eqref{spatial_regularity_proof} yields
\begin{equation*}
\begin{split}
        \mathbb{E}\left[(N(t,x)-N(t,y))^2\right]&\leq C\mathcal{C}_2^2\int_0^t\int_{M^2}\vert K_s(z)\vert\cdot[P_{t-s}(x,w)+P_{t-s}(y,w)]G_{\alpha,\rho}(w,z)dwdzds\\
        &\leq C\mathcal{C}_2^2\int_0^t\int_{M}\vert K_s(z)\vert(t-s)^{\alpha-\frac{d}{2}}dzds\\
        &\leq C\mathcal{C}_2^2d(x,y)^\xi\int_0^t(t-s)^{\alpha-\frac{d+\xi}{2}}ds\leq C\mathcal{C}_2^2d(x,y)^\xi
\end{split}
\end{equation*}
valid for $0 < \xi < \min(2\alpha-d+2,1)$ and Dalang's condition ensures $2\alpha-d+2>0$.
\end{proof}

\begin{lemma}[Temporal regularity of the noise term]\label{temporalregularity} There exists a positive constant $C$, independent of $\mathcal{C}_2$ in \eqref{hypothesis2}, such that for any $\alpha>\max\left(\frac{d-2}{2},0\right)$, $\zeta\in\left(0,\min\left(\alpha+1-\frac{d}{2},\frac{1}{2}\right)\right)$, $0\leq s\leq t\leq 1$ and $x\in M$, we have
\begin{equation*}
\mathbb{E}\left[(N(t,x)-N(s,x))^2\right]\leq\begin{cases}
    C\mathcal{C}_2^2(t-s)^{2\zeta}&\alpha>d/2\\
    C\mathcal{C}_2^2(t-s)^{\zeta} &\alpha\leq d/2.
    \end{cases}
\end{equation*}
\end{lemma}
\begin{proof}
As in the simplification used in Lemma \ref{spatialregularity}, for fix $s$, $t$ and $x$, we define, for $0<r<s$,
\begin{equation*}
L_r(z):=P_{t-r}(x-z)-P_{s-r}(x-z).
\end{equation*}
By the independence of the noise on disjoint time intervals (Definition \ref{color_noise}), we can split the difference:
\begin{equation*}
\begin{split}
\mathbb{E}[(N(t,x)-N(s,x))^2]&=\mathbb{E}\left[\left(\int_0^s\int_{M}L_r(z)\sigma(r,z,u(r,z))W(dzdr)\right.\right.\\
&\hspace{2cm}\left.\left.+\int_s^t\int_{M}P_{t-r}(x-z)\sigma(r,z,u(r,z))W(dzdr)\right)^2\right]\\
&=\mathbb{E}\left[\left(\int_0^s\int_{M}L_r(z)\sigma(r,z,u(r,z))W(dzdr)\right)^2\right]\\
&\hspace{2cm}+\mathbb{E}\left[\left(\int_s^t\int_{M}P_{t-r}(x-z)\sigma(r,z,u(r,z))W(dzdr)\right)^2\right]\\
&=:I_1+I_2
\end{split}
\end{equation*}
where the cross term vanishes.

\textbf{Estimate of $I_1$:}

As in Lemma \ref{spatialregularity},
\begin{equation*}
I_1\leq \sup_{r,v}\mathbb{E}\left[\sigma(r,v,u(r,v))^2\right]\int_0^s\int_{M^2}\vert L_r(z)\vert\cdot \vert L_r(w)\vert\cdot G_{\alpha,\rho}(w,z)dwdzdr.
\end{equation*}

\textbf{$\bullet$ Case 1: $\alpha>d/2$}
Using \eqref{hypothesis2}, Proposition \ref{covker_bound} and \eqref{heat_kernel_temporal_integral_bound},
\begin{equation*}
\begin{split}
I_1&\leq C\mathcal{C}_2^2
\int_0^s\int_{M^2}\vert L_r(z)\vert\cdot \vert L_r(w)\vert dwdzdr\\
&\leq C\mathcal{C}_2^2\int_0^s\min\left(C\log\left(\frac{t-r}{s-r}\right),2\right)^2dr\\
&\leq C\mathcal{C}_2^2(t-s)^{2\zeta}\int_0^s(s-r)^{-2\zeta}dr\leq C\mathcal{C}_2^2(t-s)^{2\zeta},
\end{split}
\end{equation*}
where we used $\min(c\log(1+x),2)\leq\max(c,2)x^\zeta$ for $x>0, 1>\zeta>0$ and $\int_0^s(s-r)^{-2\zeta}ds<\infty$ when $\zeta<\frac{1}{2}$.

\textbf{$\bullet$ Case 2: $\alpha=d/2$}
By Proposition \ref{covker_bound} and the logarithmic integrability,
\begin{equation*}
    \begin{split}
I_1&\leq C\mathcal{C}_2^2
\int_0^s\int_{M^2}\vert L_r(z)\vert\cdot \vert L_r(w)\vert \cdot(1+\log^-(d(w,z))) dwdzdr\\
&\leq C\mathcal{C}_2^2\left[(t-s)^{2\zeta}+\int_0^s\int_M\vert L_r(z)\vert\int_M\log^-(d(w,z))[P_{t-r}(x-w)+P_{s-r}(x-w)]dwdzdr\right]\\
&\leq C\mathcal{C}_2^2\left[(t-s)^{2\zeta}+(t-s)^{\zeta}\int_0^s(s-r)^{-\zeta-\frac{d}{2}(1-\frac{1}{p})}ds\right]\leq C\mathcal{C}_2^2(t-s)^{\zeta},
\end{split}
\end{equation*}
with $p<\frac{d}{d-2+2\zeta}$ to ensure convergence.

\textbf{$\bullet$ Case 3: $\alpha<d/2$}
\begin{equation*}
\begin{split}
        I_1&\leq C\mathcal{C}_2^2\int_0^s\int_{M}\vert L_r(z)\vert\int_M[P_{t-r}(x-w)+P_{s-r}(x-w)]G_{\alpha,\rho}(w,z)dwdzdr\\
        &\leq C\mathcal{C}_2^2(t-s)^{\zeta}\int_0^s(s-r)^{-\zeta}\left[(t-r)^{\alpha-\frac{d}{2}}+(s-r)^{\alpha-\frac{d}{2}}\right]dr\\
        &\leq C\mathcal{C}_2^2(t-s)^{\zeta}\int_0^s(s-r)^{-\zeta+\alpha-\frac{d}{2}}dr\leq C\mathcal{C}_2^2(t-s)^{\zeta}.
\end{split}
\end{equation*}
Note that the integral $\int_0^s(s-r)^{\zeta+\alpha-\frac{d}{2}}dr$ converges when $\zeta<\alpha-\frac{d}{2}+1$. 

\textbf{Estimate of $I_2$:}

The result is straightforward since the heat kernel is always positive. 
\begin{equation*}
\label{I2}
\begin{split}
I_2&\leq \sup_{r,v}\mathbb{E}\left[\sigma(r,v,u(r,v))^2\right]\int_s^t\int_{M^2}P_{t-r}(x,z)P_{t-r}(x,w)G_{\alpha,\rho}(w,z)dwdzdr\\
&\leq C\mathcal{C}_2^2\int_s^t\int_{M^2}P_{t-r}(x,z)P_{t-r}(x,w)G_{\alpha,\rho}(w,z)dwdzdr.\\
&\leq\begin{cases}
    C\mathcal{C}_2^2\int_s^t1dr=C\mathcal{C}_2^2(t-s)\leq C\mathcal{C}_2^2(t-s)^{2\zeta}&\alpha>d/2\\
    C\mathcal{C}_2^2\left((t-s)+\int_s^t(t-r)^{-\frac{d}{2}(1-\frac{1}{p})}dr\right)\leq C\mathcal{C}_2^2(t-s)^{\zeta}&\alpha=d/2\\
    C\mathcal{C}_2^2\int_s^t(t-r)^{\alpha-\frac{d}{2}}dr\leq C\mathcal{C}_2^2(t-s)^{\alpha-\frac{d}{2}+1}\leq C\mathcal{C}_2^2(t-s)^{\zeta}&\alpha<d/2.
\end{cases}
\end{split}
\end{equation*}
Summing the estimates for $I_1$ and $I_2$ completes the proof.
\end{proof}

The regularity estimates for the noise term yield two corresponding upper bounds on its tail probabilities, as detailed below.
\begin{lemma} There exist positive constants $C_1,C_2,C_3,C_4$, independent on $\mathcal{C}_2$ in \eqref{hypothesis2}, such that, for all $0\leq s\leq t\leq 1$, $x,y\in M$, $\alpha>\max\left(\frac{d-2}{2},0\right)$, $\zeta\in\left(0,\min\left(\alpha+1-\frac{d}{2},\frac{1}{2}\right)\right)$ and $\xi\in(0,\min\left(2\alpha-d+2,1\right))$, we have
\begin{equation}
\label{spacep}
P(\vert N(t,x)-N(t,y)\vert>\kappa)\leq\begin{cases}
C_1\exp\left(-\frac{C_2\kappa^2}{\mathcal{C}_2^2d(x,y)^{2\xi}}\right)
    &\alpha>d/2\\
    C_1\exp\left(-\frac{C_2\kappa^2}{\mathcal{C}_2^2d(x,y)^{\xi}}\right) &\alpha\leq d/2,
    \end{cases}
\end{equation}
\begin{equation}
\label{timep}
P(\vert N(t,x)-N(s,x)\vert>\kappa)\leq\begin{cases}
C_3\exp\left(-\frac{C_4\kappa^2}{\mathcal{C}_2^2(t-s)^{2\zeta}}\right)
    &\alpha>d/2\\
   C_3\exp\left(-\frac{C_4\kappa^2}{\mathcal{C}_2^2(t-s)^{\zeta}}\right) &\alpha\leq d/2.
    \end{cases}
\end{equation}
\end{lemma}
\begin{proof}
For a fixed $t$, we define
\[N_t(s,x):=\int_{[0,s]\times M} P_{t-r}(x-y)\sigma(r,y,u(r,y))W(drdy).\]
Note that $N_t(t,x)=N(t,x)$ as defined in \eqref{noise_term}, and that $N_t(s,x)$ is a continuous $\mathcal{F}_s^W$ adapted $\R$ valued martingale in $s$ since the integrand does not depend on $s$. For fixed $t,x$ and $y$, let
\[\mathcal{N}_s:=N_t(s,x)-N_t(s,y)=\int_{[0,s]\times M} K_r(z)\sigma(r,z,u(r,z))W(drdz),\]
so that $\mathcal{N}_t=N(t,x)-N(t,y)$. Since $\mathcal{N}_s$ is a continuous local martingale with $\mathcal{N}_0=0$, it can be represented as a time-changed Brownian motion, i.e.,
\[\mathcal{N}_t=B_{\langle \mathcal{N}\rangle_t}.\]
Lemma \ref{spatialregularity} provides bounds for $\langle \mathcal{N}\rangle_t$ in three regimes, and applying the reflection principle to the Brownian motion $B_{\langle \mathcal{N}\rangle_t}$, with $x$ and $y$ interchanged, yields the result in \eqref{spacep}.

For a fixed $x$, we define 
\[U_{q_1}:=\int_{[0,{q_1}]\times M} L_r(y)\sigma(r,y,u(r,y))W(drdy)\]
where $0\leq q_1\leq s$. Note that $U_{q_1}$ is a continuous $\mathcal{F}_{q_1}^W$ adapted local martingale with $U_0=0$. In addition, we define
\[V_{q_2}:=\int_{[0,{q_2}]\times M}P_{t-s-r}(x-y)\sigma(r+s,y,u(r+s,y))W(drdy)\]
where $0\leq q_2\leq t-s$. Note that $V_{q_2}$ is also a continuous $\mathcal{F}_{q_2}^W$ adapted local martingale with $V_0=0$. Thus, both $U_{q_1}$ and $V_{q_2}$ admit representations as time-changed Brownian motions, namely,
\[U_s=B_{\langle U\rangle_s}\text{~and~}V_{t-s}=B'_{\langle V\rangle_{t-s}},\]
where $B$, $B'$ are two different Brownian motions. It is clear that $N(t,x)-N(s,x)=U_s+V_{t-s}$, which leads to
\[P(N(t,x)-N(s,x)>\kappa)\leq P(U_s>\kappa/2)+P(V_{t-s}>\kappa/2).\]
Lemma \ref{temporalregularity} provides bounds for $\langle U\rangle_s$ and $\langle V\rangle_{t-s}$ in three regimes, and applying the reflection principle to two Brownian motions, with $s$ and $t$ interchanged, yields the result in \eqref{timep}.
\end{proof}

\begin{definition}
    \label{definition}
    For $n\in\N$, define the Euclidean grid
\begin{equation*}
    G_{n}:=\left\lbrace\left(\frac{j}{2^{2n}},\frac{k_1}{2^n},\dots,\frac{k_d}{2^n}\right)\in\R_+\times\R^d:j,k_1,\dots,k_d\in\Z\right\rbrace.
\end{equation*}
Two Euclidean grid points $y,y'\in G_{n}$ are “nearest neighbors” if they differ by at most one coordinate step:
\begin{enumerate}
    \item Either $y=y'$ in all but one spatial coordinate and $\vert y_i-y_i'\vert=\frac{1}{2^n}$ for that coordinate, or
    \item They differ in one step in time index and match in space.
\end{enumerate}
\end{definition}

The following lemma employs a generic dyadic chaining argument to estimate a tail probability over a small space-time region, carrying the approach of \cite{athreya2021small}. Moreover, the lemma relaxes the exponent of $\varepsilon$, allowing us to obtain exponential bounds on the probability that are valid across different regimes.
\begin{lemma}\label{larged}Given any reference point $p$ on $M$, there exist positive constants $C_5,C_6,\varepsilon_0$, independent of $\mathcal{C}_2$ in \eqref{hypothesis2} and $\varepsilon$, such that for all $\gamma,\kappa$, $0<\varepsilon<\varepsilon_0$ and $\gamma\varepsilon^4\leq 1$, we have
\begin{equation*}
P\left(\sup_{\substack{0<t\leq\gamma\varepsilon^4\\ x\in B_M\left(p,\varepsilon^2\right)}}\vert N(t,x)\vert>\kappa\varepsilon^{m}\right)\leq \frac{C_5}{1\wedge \sqrt{\gamma^d}}\exp\left(-\frac{C_6\kappa^2(\gamma\varepsilon^4)^{g(m)}}{\mathcal{C}_2^2\gamma^{m/2}}\right),
\end{equation*}
where $m\geq \min(1,2\alpha-d+2)$ and $g(m) = \frac{m-\min(1,2\alpha-d+2)}{2}$.
\end{lemma}
\begin{proof} Since $M$ is a compact smooth manifold, there exist $\varepsilon_0>0$ such that for every $p\in M$ and every $0<\varepsilon<\varepsilon_0$, the exponential map at $p$ induces a normal coordinate chart
\begin{equation*}
    \phi_p:=Id\times\exp_p^{-1}:\left[0,\gamma\varepsilon^4\right]\times B_M(p,\varepsilon^2)\to \left[0,\gamma\varepsilon^4\right]\times B_{\R^d}(0,\varepsilon^2),
\end{equation*}
which is a smooth diffeomorphism. Moreover, the Riemannian distance is uniformly bi-Lipschitz equivalent to the Euclidean distance in normal coordinates: for $x,y\in \left[0,\gamma\varepsilon^4\right]\times B_M(p,\varepsilon^2)$,
\begin{equation}
\label{biLip}
    C\vert\phi_p(x)-\phi_p(y)\vert\leq d(x,y)\leq C'\vert\phi_p(x)-\phi_p(y)\vert
\end{equation}
with Lipschitz constants $C,C'>0$ independent of $p$ and $\varepsilon$.

Fix $\gamma\geq 1$. For each $n\in\N$, we define the grid 
\[\mathbb{G}_n=\left\lbrace\left(\frac{j}{2^{2n}},\frac{k_1}{2^{n}},...,\frac{k_d}{2^{n}}\right)\in \left[0,\gamma\varepsilon^4\right]\times B_{\R^d}(0,\varepsilon^2):j,k_1,...,k_d\in\Z\right\rbrace.\]
Set $n_0=\left\lceil \log_2\left(\gamma^{-1/2}\varepsilon^{-2}\right)\right\rceil$, so that 
\begin{equation}
\label{n0}
\log_2\left(2\gamma^{-1/2}\varepsilon^{-2}\right)>n_0\geq\log_2\left(\gamma^{-1/2}\varepsilon^{-2}\right)>n_0-1.
\end{equation}
For $n< n_0$, $\mathbb{G}_n$ contains only the origin. For $n\geq n_0$, using \eqref{n0}, the cardinality of $\mathbb{G}_n$ is bounded by
\begin{equation*}
\#\mathbb{G}_n\leq C(d)\cdot\left(\gamma\varepsilon^4 2^{2n}+1\right)\cdot\left(2\varepsilon^2 2^n+1\right)^d \leq C(d)2^{(2+d)(n-n_0)}.
\end{equation*}
We choose parameters $0<\delta_1(\alpha,d)<\delta_2(\alpha,d)<m/2$, and set $\delta: = \delta_2-\delta_1$. They are chosen to satisfy the following constraint
\begin{equation}
\label{deltaconstraint}
2\zeta\wedge\xi=\begin{cases}
    m-2\delta & \min(0,d/2-1)<\alpha\leq d/2\\
    m/2-\delta & \alpha>d/2
\end{cases}
\end{equation}
for any $\zeta\in\left(0,\min\left(\alpha+1-\frac{d}{2},\frac{1}{2}\right)\right)$ and $\xi\in(0,\min\left(2\alpha-d+2,1\right))$. 

We now focus on the regime $\min(0,d/2-1)<\alpha\leq d/2$. The case $\alpha>d/2$ can be treated analogously and follows easily from the argument below. Define
\begin{equation}
\label{M}
\mathcal{M}=\frac{1-2^{-\delta_1}}{(3+d)2^{\delta n_0}}
\end{equation}
and consider the event
\[A(n,\kappa)=\left\lbrace\vert N(s)-N(q)\vert\leq \kappa \mathcal{M}\varepsilon^{m} 2^{-\delta_1n}2^{\delta_2n_0}\text{ for all nearest neighbors $\phi_p(s),\phi_p(q)\in \mathbb{G}_n$}\right\rbrace.\]
If the pair is nearest in space in the sense of Definition \ref{definition} case $\textbf{1}$, then \eqref{spacep} and \eqref{biLip} imply
\[P\left(\vert N(s)-N(q)\vert> \kappa \mathcal{M}\varepsilon^{m} 2^{-\delta_1n}2^{\delta_2n_0}\right)\leq C_1\exp\left(-\frac{C_2\kappa^2\mathcal{M}^2\varepsilon^{2m}}{2^{-n\xi}\mathcal{C}_2^2}2^{-2\delta_1n}2^{2\delta_2n_0}\right).\]
If the pair is nearest in time in the sense of Definition \ref{definition} case $\textbf{2}$, then \eqref{timep} implies
\[P\left(\vert N(s)-N(q)\vert> \kappa \mathcal{M}\varepsilon^{m}2^{-\delta_1n}2^{\delta_2n_0}\right)\leq C_3\exp\left(-\frac{C_4\kappa^2\mathcal{M}^2\varepsilon^{2m}}{2^{-2n\zeta}\mathcal{C}_2^2}2^{-2\delta_1n}2^{2\delta_2n_0}\right).\]
Therefore, a union bound over all nearest-neighbor pairs in $\mathbb{G}_n$ gives
\begin{align*}
&P(A^c(n,\kappa))\leq \sum_{\substack{\phi_p(s),\phi_p(q)\in \mathbb{G}_n\\ \text{nearest neighbors}}}P\left(\vert N(s)-N(q)\vert> \kappa \mathcal{M}\varepsilon^{m} 2^{-\delta_1n}2^{\delta_2n_0}\right)\\
&\leq C2^{(2+d)(n-n_0)}\exp\left(-\frac{C'\kappa^2\mathcal{M}^2\varepsilon^{2m}}{\mathcal{C}_2^2}2^{n(2\zeta\wedge \xi)}2^{-2\delta_1n}2^{2\delta_2n_0}\right)\\
&=C2^{(2+d)(n-n_0)}\exp\left(-\frac{C'\kappa^2\mathcal{M}^2}{\mathcal{C}_2^2}\left(\varepsilon^{2m} 2^{n_0m}\right)2^{n(2\zeta\wedge\xi)}2^{-2\delta_1n}2^{-n_0m}2^{2\delta_2n_0}\right)\\
&\leq C2^{(2+d)(n-n_0)}\exp\left(-\frac{C'\kappa^2\mathcal{M}^2}{\mathcal{C}_2^2\gamma^{m/2}}2^{(m-2\delta_2)(n-n_0)}\right),
\end{align*}
where $C,C'$ are positive constants depending only on $\alpha,d,M$. The last inequality follows from that $\varepsilon^{4} 2^{2n_0}\geq\frac{1}{2}\gamma^{-1}$ by the definition of $n_0$ in \eqref{n0}, and our choices of $\delta_1,\delta_2$ and $\delta$ in \eqref{deltaconstraint}. Let $A(\kappa)=\bigcap\limits_{n\geq n_0}A(n,\kappa)$ and we can bound $P(A^c(\kappa))$ by summing $P(A^c(n,\kappa))$ for all $n\geq n_0$,
\begin{align*}
P\left(A^c(\kappa)\right)&\leq\sum_{n\geq n_0}P\left(A^c(n,\kappa)\right)\leq \sum_{n\geq n_0}C2^{(2+d)(n-n_0)}\exp\left(-\frac{C'\kappa^2\mathcal{M}^2}{\mathcal{C}_2^2\gamma^{m/2}}2^{(m-2\delta_2)(n-n_0)}\right)\\
&\leq C_5\exp\left(-\frac{C'\kappa^2\mathcal{M}^2}{\mathcal{C}_2^2\gamma^{m/2}}\right).
\end{align*}
From definitions \eqref{n0} and \eqref{M}, it follows that
\begin{align*}
P\left(A^c(\kappa)\right)&\leq C_5\exp\left(-\frac{C'\kappa^22^{-2\delta n_0}}{\mathcal{C}_2^2\gamma^{m/2}}\right)\leq C_5\exp\left(-\frac{C_6\kappa^2(\gamma\varepsilon^4)^\delta}{\mathcal{C}_2^2\gamma^{m/2}}\right)
\end{align*}
for any $\delta=\frac{m-2\zeta\wedge\xi}{2}$. Taking the infimum over admissible $\delta$ yields the optimal bound
\begin{align*}
P\left(A^c(\kappa)\right)&\leq \inf_{\delta}C_5\exp\left(-\frac{C_6\kappa^2(\gamma\varepsilon^4)^\delta}{\mathcal{C}_2^2\gamma^{m/2}}\right)\leq
C_5\exp\left(-\frac{C_6\kappa^2(\gamma\varepsilon^4)^{\frac{m-\min(1,2\alpha-d+2)}{2}}}{\mathcal{C}_2^2\gamma^{m/2}}\right).
\end{align*}

Now for any point $(t,x)$ in $\left[0,\gamma\varepsilon^4\right]\times B_{\R^d}(0,\varepsilon^2)$, whose normal coordinate is in a grid $\mathbb{G}_n$ for some $n \geq n_0$, a standard chaining argument similar to page 128 of \cite{dalang2009minicourse} yields a path from the $(0,p)$ to $(t,x)$ as $(0,p)= p_0,p_1,...,p_l = (t,x)$ such that the normal coordinate of each pair $p_jp_{j+1}$ is the nearest neighbor in some grid $\mathbb{G}_k, n_0\leq k \leq n$. For each grid, at most $3$ temporal and $d$ spatial increments occur. On the event $A(\kappa)$, we obtain
\begin{align*}
\vert N(t,x)\vert\leq \sum_{j=0}^{l-1}\vert N(p_j)-N(p_{j+1})\vert\leq (3+d)\sum_{n\geq n_0}\kappa \mathcal{M}\varepsilon^{m} 2^{-\delta_1n}2^{\delta_2n_0}\leq\kappa\varepsilon^{m} .
\end{align*}
Points in $\bigcup\limits_n\mathbb{G}_n$ are dense in $[0,\gamma\varepsilon^4]\times B_{\R^d}(0,\varepsilon^2)$, and we may extend $N(t,x)$ to a continuous version. Hence, for $\gamma\geq 1$,
\[
P\left(\sup_{\substack{0\leq t\leq\gamma\varepsilon^4\\ x\in B_M\left(p,\varepsilon^2\right)}}\vert N(t,x)\vert>\kappa\varepsilon^{m}\right)\leq P(A^c(\kappa))\leq  C_5\exp\left(-\frac{C_6\kappa^2(\gamma\varepsilon^4)^{\frac{m-\min(1,2\alpha-d+2)}{2}}}{\mathcal{C}_2^2\gamma^{m/2}}\right).
\]
For $0<\gamma<1$, the time interval $\left[0,\gamma\varepsilon^4\right]$ is shorter relative to the spatial scale, so a direct application of the previous grid method is not optimal. We therefore cover the geodesic ball $B_M(p,\varepsilon^2)$ by $\frac{C}{\sqrt{\gamma^d}}$ pieces of smaller ball with radius $\sqrt{\gamma}\varepsilon^2$. If the supremum over the whole region is large, then it must be large in at least one sub-ball, so a union bound implies that
\[P\left(\sup_{\substack{0\leq t\leq\gamma\varepsilon^4\\ x\in B_M\left(p,\varepsilon^2\right)}}\vert N(t,x)\vert>\kappa\varepsilon^{m}\right)\leq \frac{C}{\sqrt{\gamma^d}} P\left(\sup_{\substack{0\leq t\leq\gamma\varepsilon^4\\ x\in B_M\left(p,\sqrt{\gamma}\varepsilon^2\right)}}\vert N(t,x)\vert>\kappa\varepsilon^{m}\right)\]
\[=\frac{1}{\sqrt{\gamma^d}}P\left(\sup_{\substack{0\leq t\leq(\sqrt{\gamma}\varepsilon^2)^2\\ x\in B_M\left(p,\sqrt{\gamma}\varepsilon^2\right)}}\vert N(t,x)\vert>\frac{\kappa}{\gamma^{m/4}}\left(\gamma^{1/4}\varepsilon\right)^{m}\right)\leq \frac{C_5}{\sqrt{\gamma^d}}\exp\left(-\frac{C_6\kappa^2(\gamma\varepsilon^4)^{g(m)}}{ \mathcal{C}_2^2\gamma^{m/2}}\right).\]
Combining both cases completes the proof.
\end{proof}

\begin{remark}
\label{largeremark}
If $m$ is sufficiently large so that $g(m)>0$, then letting $\varepsilon\downarrow0$ only yields a trivial upper bound, and no meaningful exponential tail estimate can be expected in this regime. Nevertheless, a nontrivial estimate may still be recovered when the noise bound $\mathcal{C}_2$ is sufficiently small.
\end{remark}

\section{Proof of Proposition \ref{prop}}
\label{proofprop}
\subsection*{Proof of Proposition \ref{prop}(a)} Recall the definitions of $F_n$ in \eqref{Fn} and $f(t_1)$ in \eqref{f_t}. By the Markov property of the solution $u(t,\cdot)$ (see Theorem 9.14, p.248 of \cite{da2014stochastic}), we have
\begin{equation*}
P\left(F_{j}\vert\sigma\{u(t_i,\cdot)\}_{0\leq i<j}\right)=P\left(F_{j}\vert u(t_{j-1},\cdot)\right).
\end{equation*}
If we can show that $P\left(F_{j}\vert u(t_{j-1},\cdot)\right)$ admits a uniform bound of the form independent of $j$, then the same bound holds for $P\left(F_{j}\Big\vert \bigcap\limits_{k=0}^{j-1}F_{k}\right)$, since this conditional probability is evaluated under a realization of ${u(t_k,\cdot)}_{0\le k<j}$. Consequently, it suffices to establish
\begin{equation}
\label{F1}
P(F_1)\leq \begin{cases}
    \textbf{C}_4\exp\left(-\frac{\textbf{C}_5}{\varepsilon^2}\right)&(d-1)/2<\alpha\leq d/2\\
    \textbf{C}_4\exp\left(-\frac{\textbf{C}_5}{t_1^{\min\left(\frac{d-2\alpha}{2d},\frac{2\alpha+2-d}{2}\right)}}\right)&\max(0,d/2-1)<\alpha\leq(d-1)/2,
    \end{cases}
\end{equation}
where $\textbf{C}_4$, $\textbf{C}_5$ are independent of $u_0$, under the assumption that $\vert u(t,x)\vert < \sqrt{f(t_1)}$ for all $(t,x)\in R_{0,J}$ as defined in \eqref{Rij}. 

To obtain a uniform bound for $P(F_1)$, we introduce the truncated function
\[\bar{f}_{t_1}(x)=\begin{cases}
x & \vert x\vert\leq \sqrt{f(t_1)}\\
\frac{x}{\vert x\vert}\cdot \sqrt{f(t_1)} & \vert x\vert> \sqrt{f(t_1)}.
\end{cases}\]
By construction, $|\bar{f}_{t_1}(x)|\leq \sqrt{f(t_1)}$. Consider the following two SPDEs with common initial condition $u_0$:
\begin{equation}
\begin{split}
\label{shef}
&\partial_tv(t,x)=\frac{1}{2}\Delta_Mv(t,x)+\sigma(t,x,\bar{f}_{t_1}(v(t,x))) \dot{W}(t,x),\\
&\partial_tv_g(t,x)=\frac{1}{2}\Delta_Mv_g(t,x)+\sigma(t,x,\bar{f}_{t_1}(u_0(x))) \dot{W}(t,x).
\end{split}
\end{equation}
Here $v_g$ is a Gaussian random field. We decompose $v(t,x)$ into
\[v(t,x)=v_g(t,x)+D(t,x),\]
where
\[D(t,x)=\int_{[0,t]\times M} P_{t-s}(x,y)[\sigma(s,y,\bar{f}_{t_1}(v(s,y)))-\sigma(s,y,\bar{f}_{t_1}(u_0(y)))]W(dsdy).\]
By the Lipschitz continuity of $\sigma$ in its third variable (Assumption \eqref{hypothesis1}),
\begin{equation}
\label{sigmadiff1}
\vert \sigma(s,y,\bar{f}_{t_1}(v(s,y)))-\sigma(s,y,\bar{f}_{t_1}(u_0(y)))\vert\leq \mathcal{D}\vert \bar{f}_{t_1}(v(s,y))-\bar{f}_{t_1}(u_0(y))\vert\leq 2\mathcal{D}\sqrt{f(t_1)}.
\end{equation}

We recall that $R_{i,j}$ in \eqref{Rij} and define a new sequence of events,
\[H_j=\left\lbrace\vert v(t,x)\vert\leq \sqrt{f(t_1)},\forall (t,x)\in R_{1,j}\setminus R_{1,j-1}\right\rbrace.\]
Clearly, by the properties of $\bar f_{t_1}$ and \eqref{Fn},
\[F_{1}=\bigcap_{j=1}^{J}H_j.\]
Introduce auxiliary events
\[A_j=\left\lbrace\vert v_g(t,x)\vert\leq 2\sqrt{f(t_1)},\forall (t,x)\in R_{1,j}\setminus R_{1,j-1}\right\rbrace\]
and
\[B_j=\left\lbrace\vert D(t,x)\vert>\sqrt{f(t_1)}, \text{~for some~}(t,x)\in R_{1,j}\setminus R_{1,j-1}\right\rbrace.\]
It is straightforward to check that $H_j^c\supset A_j^c\cap B_j^c$, which implies
\begin{equation}
\label{sumprob}
\begin{split}
P(F_{1})&= P\left(\bigcap_{j=1}^{J}H_j\right)\leq P\left(\bigcap_{j=1}^{J}[A_j\cup B_j]\right)\leq P\left(\left(\bigcap_{j=1}^{J}A_j\right)\bigcup\left(\bigcup_{j=1}^{J}B_j\right)\right)\\
&\leq P\left(\bigcap_{j=1}^{J}A_j\right)+\sum_{j=1}^{J} P(B_j),
\end{split}
\end{equation}
where the second inequality can be shown by induction. 

For $j\geq1$,
\begin{equation*}
\begin{split}
B_j&\subseteq \left\lbrace\sup_{(t,x)\in R_{1,j}\setminus R_{1,j-1}}\vert D(t,x)\vert>\sqrt{f(t_1)}\right\rbrace\\
&\subseteq\bigcup_{(t,x)\in R_{1,j}\setminus R_{1,j-1}}\left\lbrace\sup_{\substack{0\leq s\leq c_0\varepsilon^4\\ y\in B_M(x,\varepsilon^2)}}\vert D(s,y)\vert>\sqrt{f(t_1)}\right\rbrace,\\
\end{split}
\end{equation*}
which yields a union bound 
\begin{equation}
    \label{probB}
    \sum_{j=1}^{J}P(B_j)\leq\frac{C}{\varepsilon^{2d}}P\left(\sup_{\substack{0\leq s\leq c_0\varepsilon^4\\ y\in B_M(x_0,\varepsilon^2)}}\vert D(s,y)\vert>\sqrt{f(t_1)}\right).
\end{equation}

\textbf{Case 1: $\max(0,d/2-1)<\alpha\leq(d-1)/2$}

In this case, $m=2\alpha-d+2$ in Lemma \ref{larged}. By \eqref{sigmadiff1} and Remark \ref{largeremark},
\begin{equation*}
P\left(\sup_{\substack{0\leq s\leq c_0\varepsilon^4\\ y\in B_M(x,\varepsilon^2)}}\vert D(s,y)\vert>\sqrt{f(t_1)}\right)\leq \frac{C_5}{1\wedge \sqrt{{c_0}^d}}\exp\left(-\frac{C_6}{4\mathcal{D}^2t_1^{\alpha+1-d/2}}\right).
\end{equation*}
Hence,
\begin{equation*}
\sum_{j=1}^{J}P(B_j)
\leq\frac{C(d)}{\varepsilon^{2d}}\cdot\frac{C_5}{1\wedge \sqrt{{c_0}^d}}\exp\left(-\frac{C_6}{4\mathcal{D}^2t_1^{\alpha+1-d/2}}\right).
\end{equation*}

\textbf{Case 2: $(d-1)/2<\alpha<d/2$}

In this regime, $m=1$ in Lemma \ref{larged}.
\begin{equation*}
P\left(\sup_{\substack{0\leq s\leq c_0\varepsilon^4\\ y\in B_M(x,\varepsilon^2)}}\vert D(s,y)\vert>\sqrt{f(t_1)}\right)\leq \frac{C_5}{1\wedge \sqrt{{c_0}^d}}\exp\left(-\frac{C_6t_1^{\alpha+(d-1)/2}}{4\mathcal{D}^2t_1^{\alpha+1-d/2}}\right),
\end{equation*}
which yields
\begin{equation*}
\sum_{j=1}^{J}P(B_j)\leq\frac{C(d)}{\varepsilon^{2d}}\cdot\frac{C_5}{1\wedge \sqrt{{c_0}^d}}\exp\left(-\frac{C_6}{4\mathcal{D}^2t_1^{1/2}}\right).
\end{equation*}

\textbf{Case 3: $\alpha=d/2$}

A similar computation produces
\begin{equation*}
    \begin{split}
        P\left(\sup_{\substack{0\leq s\leq c_0\varepsilon^4\\ y\in B_M(x,\varepsilon^2)}}\vert D(s,y)\vert>\sqrt{f(t_1)}\right)&\leq P\left(\sup_{\substack{0\leq s\leq c_0\varepsilon^4\\ y\in B_M(x,\varepsilon^2)}}\vert D(s,y)\vert>\sqrt{c_0\ln(1/c_0)\varepsilon^4}\right)\\
        &\leq \frac{C_5}{1\wedge \sqrt{{c_0}^d}}\exp\left(-\frac{C_6\ln(1/c_0)}{4\mathcal{D}^2t_1^{1/2}\ln(1/t_1)}\right),
    \end{split}
\end{equation*}
and therefore
\begin{equation*}
\sum_{j=1}^{J}P(B_j)\leq\frac{C(d)}{\varepsilon^{2d}}\cdot\frac{C_5}{1\wedge \sqrt{{c_0}^d}}\exp\left(-\frac{C_6\ln(1/c_0)}{4\mathcal{D}^2t_1^{1/2}\ln(1/t_1)}\right).
\end{equation*}

Note that in the equation defining $v_g$, the coefficient $\sigma$ is deterministic; consequently, $v_g$ is a Gaussian random field. The following lemma establishes bounds on the variance of the noise term $N(t_1,x)$, as well as an upper bound on the covariance between $N(t_1,x)$ and $N(t_1,y)$.
\begin{lemma}
    Let $N(t,x)$ and $N(t,y)$ be noise terms corresponding to a deterministic coefficient $\sigma(t,x,u)=\sigma(t,x)$, and let $t\downarrow0$. Then there exist positive constants $C_7,C_8,C_9$ such that
    \begin{equation}
    \label{varbound}
        C_7f(t)\leq \text{Var}[N(t,x)]\leq C_8f(t),
    \end{equation}
    where $f(t)$ is defined in \eqref{f_t}. Moreover,
    \begin{equation}
    \label{covbound}
        \text{Cov}[N(t,x)N(t,y)]\leq\begin{cases}
        C_9t & \alpha>d/2\\
        C_9t(1+\log^-(d(x,y))) & \alpha=d/2\\
        C_9td(x,y)^{2\alpha-d} & \max(0,d/2-1)<\alpha<d/2.
        \end{cases}
    \end{equation}
\end{lemma}
\begin{proof}
    By Definition \ref{color_noise}, together with the uniform upper bound on $\sigma$ in \eqref{hypothesis2}, the representation of $G_{\alpha,\rho}$ in Proposition \ref{covker_bound}, and the heat kernel bounds in \eqref{heatkernelbound1},\eqref{heatkernelbound2} applied after splitting the $r$-integration, we obtain
    \begin{equation*}
    \begin{split}
        \text{Var}[N(t,x)] &= \int_{[0,t]\times M^2}P_{t-s}(x,y)P_{t-s}(x,z)\sigma(s,y)\sigma(s,z)G_{\alpha,\rho}(y,z)dydzds\\
        &\leq C\mathcal{C}_2^2\left[t+ \int_0^t\int_0^\infty r^{\alpha-1}\left(P_{2s+r}(x,x)-\frac{1}{m_0}\right)drds\right]\\
        &\leq C\mathcal{C}_2^2\left[t+\int_0^t\int_{1-2s}^\infty r^{\alpha-1}e^{-C(2s+r)}drds +\int_0^t\int_0^1r^{\alpha-1}(2s+r)^{-d/2}drds\right]\\
        & \leq C\mathcal{C}_2^2\left[t+t^{\alpha+1-d/2}\int_0^1\int_0^{\frac{1}{t}}w^{\alpha-1}(2v+w)^{-d/2}dwdv \right].
    \end{split}
    \end{equation*}
    The behavior of the remaining integral as $t\downarrow0$ depends on $\alpha-d/2$:

    $\bullet$ If $\alpha>d/2$,  
    \begin{equation*}
        t^{\alpha+1-d/2}\int_0^1\int_0^{\frac{1}{t}}w^{\alpha-1}(2v+w)^{-d/2}dwdv \leq t^{\alpha+1-d/2}\int_0^1\int_0^{\frac{1}{t}}w^{\alpha-1-d/2}dwdv=t.
    \end{equation*}

    $\bullet$ If $\alpha=d/2$, one checks in all subcases ($\alpha=1$, $\alpha>1$, or $\alpha=1/2$) that
    \begin{equation*}
        t^{\alpha+1-d/2}\int_0^1\int_0^{\frac{1}{t}}w^{\alpha-1}(2v+w)^{-d/2}dwdv\overset{(\alpha=1)}{=}t\int_0^1\int_0^{\frac{1}{t}}(2v+w)^{-1}dwdv\leq Ct\ln(1/t);
    \end{equation*}
    \begin{equation*}
    \begin{split}
        t^{\alpha+1-d/2}&\int_0^1\int_0^{\frac{1}{t}}w^{\alpha-1}(2v+w)^{-d/2}dwdv\overset{(\alpha>1)}{=}t\int_0^1\int_0^{\frac{1}{t}}w^{\alpha-1}(2v+w)^{-\alpha}dwdv\\
        &=t\int_0^1\int_0^{\frac{1}{1+2vt}}\frac{x^{\alpha-1}}{1-x}dxdv\leq t\int_0^1\int_0^{\frac{1}{1+2vt}}\frac{1}{1-x}dxdv\leq Ct\ln(1/t);
    \end{split}
    \end{equation*}
    \begin{equation*}
    \begin{split}
        t^{\alpha+1-d/2}&\int_0^1\int_0^{\frac{1}{t}}w^{\alpha-1}(2v+w)^{-d/2}dwdv\overset{(\alpha=1/2)}{=}t\int_0^1\int_0^{\frac{1}{t}}w^{-\frac{1}{2}}(2v+w)^{-\frac{1}{2}}dwdv\\
        &=t\int_0^1\ln\left(\frac{1+\sqrt{\frac{1}{1+2vt}}}{1-\sqrt{\frac{1}{1+2vt}}}\right)dv\leq t\ln\left(\frac{2}{1-\sqrt{\frac{1}{1+2t}}}\right)\leq Ct\ln(1/t).
    \end{split}
    \end{equation*}

    $\bullet$ If $\alpha<d/2$, the integral converges when $t\downarrow0$,
    \begin{equation*}
        t^{\alpha+1-d/2}\int_0^1\int_0^{\frac{1}{t}}w^{\alpha-1}(2v+w)^{-d/2}dwdv \leq Ct^{\alpha+1-d/2}.
    \end{equation*}
    
Using the uniform lower bound on $\sigma$, restricting to small $r$ and $t<1/4$, and proceeding as above, we obtain
     \begin{equation*}
    \begin{split}
        \text{Var}[N(t,x)] &= \int_{[0,t]\times M^2}P_{t-s}(x,y)P_{t-s}(x,z)\sigma(s,y)\sigma(s,z)G_{\alpha,\rho}(y,z)dydzds\\
        &\geq C\mathcal{C}_1^2\left[t+\int_0^t\int_0^\infty r^{\alpha-1}\left(P_{2s+r}(x,x)-\frac{1}{m_0}\right)drds\right]\\
        &\geq C\mathcal{C}_1^2\left[t-\int_0^t\int_{1-2s}^\infty r^{\alpha-1}e^{-C(2s+r)}drds +\int_0^t\int_0^{\frac{1}{2}}r^{\alpha-1}(2s+r)^{-d/2}drds\right]\\
        & \geq C\mathcal{C}_1^2\left[t+t^{\alpha+1-d/2}\int_0^1\int_0^{\frac{1}{2t}}w^{\alpha-1}(2v+w)^{-d/2}dwdv \right].
    \end{split}
    \end{equation*}

    For $\alpha\ge d/2$, the last integral produces a contribution of order $t$ (or $t\log(1/t)$ when $\alpha=d/2$):
    \begin{equation*}
    \begin{split}
         t^{\alpha+1-d/2}&\int_0^1\int_0^{\frac{1}{2t}}w^{\alpha-1}(2v+w)^{-d/2}dwdv \geq t^{\alpha+1-d/2}\int_0^1\int_0^{\frac{1}{2t}}w^{\alpha-1}(2+w)^{-d/2}dwdv\\
         &\geq t^{\alpha+1-d/2}\int_{1}^{\frac{1}{2t}}w^{\alpha-1}(2+w)^{-d/2}dw\geq Ct^{\alpha+1-d/2}\int_{1}^{\frac{1}{2t}}w^{\alpha-1-d/2}dw,
    \end{split}
    \end{equation*}
    and the lower bound follows. When $\max(0,d/2-1)<\alpha<d/2$, it converges and yields a contribution of order $t^{\alpha+1-d/2}$:
    \begin{equation*}
        t^{\alpha+1-d/2}\int_0^1\int_0^{\frac{1}{2t}}w^{\alpha-1}(2v+w)^{-d/2}dwdv \geq Ct^{\alpha+1-d/2},
    \end{equation*}
    which completes the proof of \eqref{varbound}.

    The covariance satisfies
    \begin{equation*}
    \begin{split}
        \text{Cov}[N(t,x)N(t,y)] &= \int_{[0,t]\times M^2}P_{t-s}(x,w)P_{t-s}(y,z)\sigma(s,w)\sigma(s,z)G_{\alpha,\rho}(w,z)dwdzds\\
        &\leq C\mathcal{C}_2^2\left[t+ \int_0^t\int_0^\infty r^{\alpha-1}\left(P_{2s+r}(x,y)-\frac{1}{m_0}\right)drds\right]\\
        &\leq C\mathcal{C}_2^2\left[t+\int_0^t\int_{1-2s}^\infty r^{\alpha-1}e^{-C(2s+r)}drds+\int_0^t\int_0^{1-2s}r^{\alpha-1}drds\right.\\
        &\quad\left.+\int_0^t\int_0^{1-2s}r^{\alpha-1}(2s+r)^{-d/2}\exp\left(-\frac{d(x,y)^2}{4s+2r}\right)drds\right].
    \end{split}
    \end{equation*}
    Denote the last term by
    \begin{equation*}
        I(t) = \int_0^t\int_0^{1-2s}r^{\alpha-1}(2s+r)^{-d/2}\exp\left(-\frac{d(x,y)^2}{4s+2r}\right)drds.
    \end{equation*}
After the change of variables $u=2s+r$ and splitting the integration domain, one obtains
    \begin{equation*}
        \begin{split}
             I(t)&=C\int_0^t\int_{2s}^{1}(u-2s)^{\alpha-1}u^{-d/2}\exp\left(-\frac{d(x,y)^2}{2u}\right)duds\\
             &=C\int_0^{1}u^{-d/2}\exp\left(-\frac{d(x,y)^2}{2u}\right)\left(\int_{0}^{\min(t,u/2)}(u-2s)^{\alpha-1}ds\right)du\\
             &\leq C\left[\int_0^{4t}u^{\alpha-d/2}\exp\left(-\frac{d(x,y)^2}{2u}\right)du\right.\\
             &\quad \left.+\int_{4t}^1u^{-d/2}\exp\left(-\frac{d(x,y)^2}{2u}\right)\left(\int_{0}^{t}(u-2s)^{\alpha-1}ds\right)du\right]\\
            &:=I_1(t)+I_2(t).
        \end{split}
    \end{equation*}
$I_1$ can be bounded by
\begin{equation*}
    I_1(t)\leq Cd(x,y)^{2\alpha-d+2}\int_{\frac{d(x,y)^2}{8t}}^\infty v^{-\alpha+d/2-2}e^{-v}dv\leq C\frac{t^{\alpha-d/2+2}}{d(x,y)^2}\exp\left(-\frac{d(x,y)^2}{8t}\right)\leq Ct.
\end{equation*}
 Since $u-2t\in(2t,u)\subset(u/2,u)$ when $u\in(4t,1)$, by the mean value theorem, we obtain
 \begin{equation*}
     \int_{0}^{t}(u-2s)^{\alpha-1}ds\leq Ctu^{\alpha-1},
 \end{equation*}
and $I_2$ can be bounded by
\begin{equation*}
    I_2(t)\leq Ct\int_0^{1}u^{\alpha-1-d/2}\exp\left(-\frac{d(x,y)^2}{2u}\right)du.
\end{equation*}
    Applying Proposition \ref{covker_bound} (Lemma 2.9 of \cite{brosamler1983laws}) yields
    \begin{equation*}
        I_2(t)\leq\begin{cases}
    Ct &\alpha>d/2\\
    Ct(1+\log^-(d(x,y))) &\alpha=d/2\\
    Ct d(x,y)^{2\alpha-d} &\alpha<d/2.
\end{cases}
    \end{equation*}
Combining these estimates proves \eqref{covbound} and completes the proof.
\end{proof}

We now turn to computing an upper bound for $P\left(\bigcap\limits_{j=1}^{J}A_j\right)$ in \eqref{sumprob}. To this end, define a sequence of events involving $v_g$:
\[I_j=\left\lbrace\vert v_g(t,x)\vert\leq 2\sqrt{f(t_1)}, \forall(t,x)\in R_{1,j}\right\rbrace\text{ for $j\geq 1$ and $I_{0}=\Omega$.}\]
Then, we can express $P\left(\bigcap\limits_{j=1}^{J}A_j\right)$ as a product of conditional probabilities:
\begin{equation}
\label{Acond}
P\left(\bigcap_{j=1}^{J}A_j\right)=P(I_{J})=P(I_{0})\prod_{j=1}^{J}\frac{P(I_j)}{P(I_{j-1})}=\prod_{j=1}^{J}P(I_j\vert I_{j-1}).
\end{equation}
Let $\mathcal{G}_j$ denote the $\sigma-$algebra generated by 
\[N_{(t_1)}(t,x):=\int_{[0,t]\times M} P_{t-s}(x,y)\sigma(s,y,\bar{f}_{t_1}(u_0(y)))W(dsdy),~~(t,x)\in R_{1,j},\]
which represents the noise term in $v_g(t_1,x)$ (see \eqref{shef}). Once we establish a uniform bound for $P\left(I_j\vert \mathcal{G}_{j-1}\right)$, the same bound holds for $P\left(I_j\vert I_{j-1}\right)$. Since $\sigma(s,y,\bar{f}_{t_1}(u_0(y)))$ is deterministic and uniformly bounded, Lemma \ref{varbound} gives
\begin{equation*}
\textrm{Var}\left[N_{(t_1)}(t_1,x)\right]\geq C_7 f(t_1).
\end{equation*}
For $(t,x)\in R_{1,j}\setminus R_{1,j-1}$, one may decompose
\begin{equation}
\label{vgdecom}
v_g(t,x)=\int_{M}P_{t}(x,y)u_0(y)dy+X+Y,
\end{equation}
where 
\begin{equation}
\label{Xdecom}
X=\E\left[N_{(t_1)}(t,x)\vert \mathcal{G}_{j-1}\right]=\sum_{(t,x)\in R_{1,j-1}}\eta^{(j)}(t,x)N_{(t_1)}(t,x)
\end{equation}
is a Gaussian variable expressed as a linear combination of generators of $\mathcal{G}_{j-1}$, and
\begin{equation*}
Y=N_{(t_1)}(t,x)-X.
\end{equation*}
By construction, $Y$ is independent of $\mathcal{G}_{j-1}$, so that
\begin{align*}
&\textrm{Var}(Y\vert\mathcal{G}_{j-1})=\textrm{Var}[N_{(t_1)}(t,x)\vert \mathcal{G}_{j-1}]=\textrm{Var}(Y).
\end{align*}
Applying Anderson's inequality \cite{anderson1955integral}, for a Gaussian random variable $Z \sim N(\mu,\sigma^2)$ and any $a >0$, the probability $P(\vert Z\vert \leq a)$ is maximized at $\mu = 0$. Hence
\begin{equation}
\label{probA}
\begin{split}
P\left(I_j\vert \mathcal{G}_{j-1}\right)&\leq P\left(\vert v_g(t,x)\vert\leq 2\sqrt{f(t_1)}, (t,x)\in R_{1,j}\setminus R_{1,j-1}\bigg| \mathcal{G}_{j-1}\right)\\
&\leq P\left(\left\vert Z'\right\vert\leq \frac{2\sqrt{f(t_1)}}{\sqrt{\textrm{Var}[N_{(t_1)}(t,x)\vert \mathcal{G}_{j-1}]}}\right)
\end{split}
\end{equation}
where $Z'\sim N(0,1)$. Let's use the notation $\textrm{SD}$ to denote the standard deviation of a random variable. Using the Minkowski inequality applied to \eqref{vgdecom} and \eqref{Xdecom}, we have
\[\textrm{SD}[N_{(t_1)}(t,x)]\leq \textrm{SD}(X)+\textrm{SD}(Y)\]
and
\[\textrm{SD}(X)\leq \sum_{(t,x)\in R_{1,j-1}}\left\vert\eta^{(j)}(t,x)\right\vert\cdot \textrm{SD}[N_{(t_1)}(t,x)].\]
By choosing the coefficients to satisfy
\[\sum_{(t,x)\in R_{1,j-1}}\left\vert\eta^{(j)}(t,x)\right\vert<\frac{\sqrt{C_7}}{2\sqrt{C_8}},\]
and applying Lemma \ref{varbound}, we obtain
\begin{equation*}
\begin{split}
\textrm{SD}(X)&\leq \sum_{(t,x)\in R_{1,j-1}}\left\vert\eta^{(j)}(t,x)\right\vert\cdot \textrm{SD}[N_{(t_1)}(t,x)]\\
&\leq \left(\sum_{(t,x)\in R_{1,j-1}}\left\vert\eta^{(j)}(t,x)\right\vert\right)\cdot\sup_{(t,x)\in R_{1,j-1}}\textrm{SD}[N_{(t_1)}(t,x)]\\
&<\frac{\sqrt{C_7}}{2\sqrt{C_8}}\cdot \sqrt{C_8f(t_1)}=\frac{\sqrt{C_7f(t_1)}}{2}.
\end{split}
\end{equation*}
Therefore, $\textrm{SD}(Y)$ can be bounded below by
\begin{equation*}
\textrm{SD}(Y)\geq \textrm{SD}(N_{(t_1)}(t,x)) - \textrm{SD}(X)>\frac{\sqrt{C_7f(t_1)}}{2}.
\end{equation*}
Thus, \eqref{probA} yields a uniform bound:
\begin{equation*}
\begin{split}
P(I_j\vert \mathcal{G}_{j-1})&\leq P\left(\vert Z'\vert\leq \frac{2\sqrt{f(t_1)}}{\sqrt{\textrm{Var}[N_{(t_1)}(t,x)\vert \mathcal{G}_{j-1}]}}\right)\\
&\leq P\left(\vert Z'\vert\leq \frac{4\sqrt{f(t_1)}}{\sqrt{C_7 f(t_1)}}\right)\\
&=P\left(\vert Z'\vert\leq C'\right)<1,
\end{split}
\end{equation*}
where $C'$ depends only on $d$, $\alpha$, and $M$. Combining this with \eqref{Acond} gives
\begin{equation}
\label{probAbound}
P\left(\bigcap_{j=1}^{J}A_j\right)\leq C^{C'\varepsilon^{-2}}=\exp\left(-\frac{C}{\varepsilon^2}\right),
\end{equation}
where $C$ depends only on $d$, $\alpha$, and $M$. The following lemma shows how to select $c_0$ to ensure $\sum\limits_{(t,x)\in R_{1,j-1}}\vert\eta^{(j)}(t,x)\vert\leq \frac{\sqrt{C_7}}{2\sqrt{C_8}}$, thus completing the proof.

\begin{lemma}\label{coeffbound}
For a given $\varepsilon>0$, one can choose $c_0>0$ in \eqref{c0} such that
\begin{equation*}
\sum_{(t,x)\in R_{1,j-1}}\vert\eta^{(j)}(t,x)\vert\leq \frac{\sqrt{C_7}}{2\sqrt{C_8}}.
\end{equation*}
\end{lemma}
\begin{proof}
Let $X$ and $Y$ be defined as in \eqref{vgdecom} and \eqref{Xdecom}. Since $Y$ is independent of $\mathcal{G}_{j-1}$, we have, for all $(t,x)\in R_{1,j-1}$,
\[\textrm{Cov}[Y,N_{(t_1)}(t,x)]=0,\]
and for $(t,y)\in R_{1,j}\setminus R_{1,j-1}$,
\begin{equation}
\label{noisecov}
\begin{split}
\textrm{Cov}[N_{(t_1)}(t,x),N_{(t_1)}(t,y)]&=\textrm{Cov}[N_{(t_1)}(t,x),X]\\
&=\sum_{(t,x')\in R_{1,j-1}}\eta^{(j)}(t,x') \textrm{Cov}[N_{(t_1)}(t,x),N_{(t_1)}(t,x')].
\end{split}
\end{equation}
Writing \eqref{noisecov} in matrix form gives
\begin{equation*}
\textbf{X} =\Sigma\eta,
\end{equation*}
where $\eta=\left(\eta^{(j)}(t,x)\right)_{(t,x)\in R_{1,j-1}}^T$, $\textbf{X}=\left\lbrace\textrm{Cov}[N_{(t_1)}(t,x),N_{(t_1)}(t,y)]\right\rbrace_{(t,x)\in R_{1,j-1}}^T,$
and $\Sigma$ is the covariance matrix of $\left(N_{(t_1)}(t,x)\right)_{(t,x)\in R_{1,j-1}}$. Let $\vert\vert\cdot\vert\vert_{1,1}$ denote the matrix norm induced by the $l_1$-vector norm, that is for a matrix $\textbf{A}$,
\[\vert\vert\textbf{A}\vert\vert_{1,1}:=\sup_{\textbf{x}\neq \textbf{0}}\frac{\vert\vert \textbf{Ax}\vert\vert_{l_1}}{\vert\vert\textbf{x}\vert\vert_{l_1}}.\]
It can be shown that $\vert\vert \textbf{A}\vert\vert_{1,1}=\max\limits_j\sum\limits_{i=1}^n\vert a_{ij}\vert$ (see page 259 of \cite{rao2000linear}). Then
\begin{equation*}
\vert\vert \eta\vert\vert_{l_1}=\vert\vert\Sigma^{-1}\textbf{X}\vert\vert_{l_1}\leq\vert\vert\Sigma^{-1}\vert\vert_{1,1}\vert\vert\textbf{X}\vert\vert_{l_1}.
\end{equation*}

Decompose $\Sigma=\textbf{D}\textbf{T}\textbf{D}$, where $\textbf{D}$ is a diagonal matrix with entries $\sqrt{\textrm{Var}[N_{(t_1)}(t,x)]}$, and $\textbf{T}$ is the correlation matrix with entries
\[e_{t_1}(x,x')=\frac{\textrm{Cov}[N_{(t_1)}(t,x),N_{(t_1)}(t,x')]}{\sqrt{\textrm{Var}[N_{(t_1)}(t,x)]}\cdot\sqrt{\textrm{Var}[N_{(t_1)}(t,x')]}}.\]
Thanks to Lemma \ref{varbound}, for $x\neq x'$,
\[\vert e_{t_1}(x,x')\vert\leq \begin{cases}
    \frac{C_9}{C_7} & \alpha>d/2;\\
    \frac{C_9(1+\log^-(d(x,x')))}{C_7\ln(1/t_1)} & \alpha = d/2;\\
    \frac{C_9d(x,x')^{2\alpha-d}}{C_7t_1^{\alpha-d/2}} & \alpha < d/2.
\end{cases}\]
For a fixed $x_j\in\{x_k\}$, sum over all other points:
\begin{equation*}
\begin{split}
    \sum_{x'\neq x_j}\vert e_{t_1}(x_j,x')\vert\leq\begin{cases}
CN(\varepsilon)&\alpha>d/2;\\
\frac{C}{\ln(1/t_1)}\sum\limits_{x'\in\{x_k\}}(1+\log^-(d(x_j,x')))&\alpha=d/2;\\
Ct_1^{d/2-\alpha}\sum\limits_{x'\in\{x_k\}}d(x_j,x')^{2\alpha-d}&\alpha<d/2.
\end{cases}
\end{split}
\end{equation*}

For $\alpha>d/2$, the total sum of the correlation coefficients diverges at the rate $\varepsilon^{-2d}$ in the limit $\varepsilon\downarrow0$. For the latter two cases, we partition points into dyadic shells around $x_j$:
\begin{equation*}
    A_m(x_j):=\left\lbrace x\in\{x_k\}:2^m\varepsilon^2\leq d(x_j,x)<2^{m+1}\varepsilon^2\right\rbrace, \quad m\geq 0,
\end{equation*}
with $\vert A_m(x_j)\vert\leq C2^{d(m+1)}$. When $\alpha<d/2$,
\begin{equation*}
    \begin{split}
        \sum_{x'\neq x_j}\vert e_{t_1}(x_j,x')\vert&\leq Ct_1^{d/2-\alpha}\sum_{m=0}^{m_{max}}\sum_{x\in A_m(x_j)}d(x_j,x)^{2\alpha-d}\\
        &\leq C2^dt_1^{d/2-\alpha}\varepsilon^{4\alpha-2d}\sum_{m=0}^{m_{max}}2^{2m\alpha}\\
        &\leq Ct_1^{d/2-\alpha}\varepsilon^{-2d}.
    \end{split}
\end{equation*}

When $\alpha=d/2$,
\begin{equation*}
    \begin{split}
        \sum_{x'\neq x_j}\vert e_{t_1}(x_j,x')\vert&\leq \frac{C}{\ln(1/t_1)}\sum_{m=0}^{m_{max}}\sum_{x\in A_m(x_j)}(1+\ln(1/d(x_j,x)))\\
        &\leq \frac{C2^d}{\ln(1/t_1)}\sum_{m=0}^{m_{max}}2^{md}\left(1+\ln(1/\varepsilon^2)-m\ln 2\right)\\
        &\leq \frac{C\varepsilon^{-2d}\ln(1/\varepsilon)}{\ln(1/t_1)}.
    \end{split}
\end{equation*}

Let $\textbf{A}=\textbf{I}-\textbf{T}$ with zero diagonal entries, 
\begin{equation}
\label{matrixA}
\vert\vert\textbf{A}\vert\vert_{1,1}=\max_{(t,x)\in R_{1,j-1}}\sum_{x'\neq x}\vert e_{t_1}(x,x')\vert\leq\begin{cases}
\frac{C\varepsilon^{-2d}\ln(1/\varepsilon)}{\ln(1/t_1)} & \alpha=d/2\\
    Ct_1^{d/2-\alpha}\varepsilon^{-2d} & \alpha<d/2.
\end{cases}
\end{equation}
Choose $\varepsilon>0$ such that $\vert\vert\textbf{A}\vert\vert_{1,1}<\phi<1$. For instance,
\begin{equation}
\label{selectc0}
c_0<\begin{cases}
\exp\left(C\varepsilon^{-2d}\ln(\varepsilon)\right)&\alpha=d/2\\
C\varepsilon^{\frac{8\alpha}{d-2\alpha}}&\alpha<d/2,
\end{cases}
\end{equation}
and such $c_0$ will produce
\begin{equation*}
\vert\vert\textbf{T}^{-1}\vert\vert_{1,1}=\vert\vert(\textbf{I}-\textbf{A})^{-1}\vert\vert_{1,1}\leq \frac{1}{1-\vert\vert\textbf{A}\vert\vert_{1,1}}<\frac{1}{1-\phi},
\end{equation*}
and $\vert\vert\Sigma^{-1}\vert\vert_{1,1}\leq \vert\vert\textbf{D}^{-1}\vert\vert_{1,1}\cdot\vert\vert\textbf{T}^{-1}\vert\vert_{1,1}\cdot\vert\vert\textbf{D}^{-1}\vert\vert_{1,1}\leq \frac{1}{C_7(1-\phi)f(t_1)}$. By Lemma \eqref{varbound}, $\vert\vert\textbf{X}\vert\vert_{l_1}$ can be bounded above by
\begin{equation*}
\vert\vert\textbf{X}\vert\vert_{l_1}\leq\begin{cases}
    \sum\limits_{(t,x)\in R_{1,j-1}}C_9t_1(1+\log^-(d(x,y)) & \alpha=d/2\\
    \sum\limits_{(t,x)\in R_{1,j-1}}C_9t_1d(x,y)^{2\alpha-d}&\alpha<d/2,
\end{cases}
\end{equation*}
where $(t,y)\in R_{1,j}\setminus R_{1,j-1}$. As how we compute the upper bound of $\vert\vert\textbf{A}\vert\vert_{1,1}$ in \eqref{matrixA}, we obtain
\begin{equation*}
\begin{split}
\vert\vert\eta\vert\vert_{l_1}&\leq \frac{1}{C_7(1-\phi)f(t_1)}\vert\vert\textbf{X}\vert\vert_{l_1}<\frac{\phi}{1-\phi}\rightarrow 0 \quad \text{as } \phi \downarrow 0.
\end{split}
\end{equation*}
Hence, with this choice of $c_0$,
\begin{equation*}
\sum_{(t,\textbf{x})\in R_{1,j-1}}\vert\eta^{(j)}(t,\textbf{x})\vert\leq \frac{\sqrt{C_7}}{2\sqrt{C_8}}.
\end{equation*}
\end{proof}

For $\alpha=d/2$, combining \eqref{sumprob} and \eqref{probAbound} gives
\begin{equation*}
\begin{split}
P(F_{1})&\leq \frac{C(d)C_5}{(1\wedge \sqrt{{c_0}^d})\varepsilon^{2d}}\exp\left(-\frac{C_6\ln(1/c_0)}{4\mathcal{D}^2t_1^{1/2}\ln(1/t_1)}\right)+\exp\left(-\frac{C}{\varepsilon^2}\right)\\
&\leq \frac{C}{t_1^{d/2}}\exp\left(-\frac{C_6\ln(1/c_0)}{4\mathcal{D}^2t_1^{1/2}\ln(1/t_1)}\right)+\exp\left(-\frac{C}{\varepsilon^2}\right),
\end{split}
\end{equation*}
and for a sufficiently small $\varepsilon$, this yields
\begin{align*}
P(F_1)&\leq \textbf{C}_4\exp\left(-\frac{\textbf{C}_5}{\varepsilon^2}\right).
\end{align*}

For $(d-1)/2<\alpha<d/2$, 
\begin{equation*}
\begin{split}
P(F_{1})&\leq \frac{C(d)C_5}{(1\wedge \sqrt{{c_0}^d})\varepsilon^{2d}}\exp\left(-\frac{C_6}{4\mathcal{D}^2t_1^{1/2}}\right)+\exp\left(-\frac{C}{\varepsilon^2}\right)\\
&\leq \frac{C}{t_1^{d/2}}\exp\left(-\frac{C_6}{4\mathcal{D}^2t_1^{1/2}}\right)+\exp\left(-\frac{C}{\varepsilon^2}\right).
\end{split}
\end{equation*}
Using $\varepsilon^2=\left(Ct_1\right)^{\frac{d-2\alpha}{2d}}$ from \eqref{c0}, for a small enough $\varepsilon$ and $\mathcal{D}>0$, we conclude
\begin{align*}
P(F_1)&\leq C\exp\left(-\frac{C'}{4\mathcal{D}^2t_1^{1/2}}\right)+\exp\left(-\frac{C}{\varepsilon^2}\right)\\
&\leq\textbf{C}_4\exp\left(-\frac{\textbf{C}_5}{\varepsilon^2}\right).
\end{align*}

Finally, for $\max(0,d/2-1)<\alpha\leq (d-1)/2$,
\begin{equation*}
\begin{split}
P(F_{1})&\leq \frac{C(d)C_5}{(1\wedge \sqrt{{c_0}^d})\varepsilon^{2d}}\exp\left(-\frac{C_6}{4\mathcal{D}^2t_1^{\alpha+1-d/2}}\right)+\exp\left(-\frac{C}{\varepsilon^2}\right)\\
&\leq C\exp\left(-\frac{C'}{\mathcal{D}^2t_1^{\alpha+1-d/2}}\right)+\exp\left(-\frac{C}{\varepsilon^2}\right)\\
&\leq\textbf{C}_4\exp\left(-\frac{\textbf{C}_5}{t_1^{\min\left(\frac{d-2\alpha}{2d},\frac{2\alpha+2-d}{2}\right)}}\right),
\end{split}
\end{equation*}
which completes the proof of \eqref{F1} as well as Proposition \ref{prop}(a).

\subsection*{Proof of Proposition \ref{prop}(b)} The following lemma is a key ingredient in the proof of Proposition \ref{prop}(b).
\begin{lemma}[The Gaussian correlation inequality]\label{Gaussiancorr}For any convex symmetric sets $K, L$ in $\R^d$ and any centered Gaussian measure $\mu$ on $\R^d$, we have
\[
\mu(K\cap L)\geq \mu(K)\mu(L).
\]
\end{lemma}
\begin{proof}
See in paper \cite{royen2014simple} and \cite{latala2017royen}. 
\end{proof}

By the Markov property of $u(t,\cdot)$, the behavior of the solution on the interval $[t_n,t_{n+1}]$ depends only on $u(t_n,\cdot)$ and on the noise $\dot{W}(t,x)$ restricted to $[t_n,t]\times M$. Arguing as in the proof of Proposition \ref{prop}(a) and recalling $\mathcal{C}_3$ in \eqref{c3}, it suffices to show that
\begin{equation}
\label{E0}
P\left(E_{0}\right)\geq  \textbf{C}_6\exp\left(-\frac{\textbf{C}_7}{t_1^{d/2}}\right),
\end{equation}
where the constants $\textbf{C}_6$, $\textbf{C}_7$ are independent of $u_0$, provided that $\vert u_0(x)\vert\leq \frac{\mathcal{C}_3}{3}t_1^{h/4}$ for all $x\in M$ with $h=\min(1,2\alpha-d+2)$ .

We prove \eqref{E0} in two steps. First, we establish the estimate in the Gaussian case. We then extend the analysis to the non-Gaussian setting and demonstrate that the approximation error between the Gaussian and non-Gaussian cases can be effectively controlled by an appropriate choice of the time intervals over which the coefficients are frozen.

\subsection*{Deterministic $\sigma$}
We begin with the Gaussian case, where $\sigma(t,x,u)$ is deterministic and depends only on $(t,x)$. For $n \geq 0$, we define the events
\begin{equation}
\label{Dsequence}
D_n=\left\lbrace\vert u(t_{n+1},x)\vert\leq \frac{\mathcal{C}_3}{6}t_1^{h/4},\text{and}~\vert u(t,x)\vert\leq \frac{2\mathcal{C}_3}{3}t_1^{h/4},~\forall t\in[t_n,t_{n+1}],x\in M\right\rbrace.
\end{equation}
Let
\[P_t(u_0)(x):=P_{t}(\cdot,\cdot)*u_0(x)=\int_{M}P_{t}(x,y)u_0(y)dy,\]
then
\begin{equation}
\label{u0convolute}
\vert P_t(u_0)(x)\vert\leq \sup_{x\in M}\vert u_0(x)\vert\leq \frac{\mathcal{C}_3}{3}t_1^{h/4}.
\end{equation}

Suppose $\theta:[0,{t_1}]\times\Omega\to\mathcal{H}^{\alpha,\rho}$ is an $\mathcal{F}_t$-progressively measurable process satisfying:
\begin{enumerate}
    \item Integrability:
    \begin{equation*}
        \int_0^{t_1}\Vert \theta_t\Vert^2_{\mathcal{H}^{\alpha,\rho}}dt<\infty\quad a.s.
    \end{equation*}
    \item Novikov's condition:
    \begin{equation*}
        \E\left[\exp\left(\frac{1}{2}\int_0^{t_1}\Vert \theta_t\Vert^2_{\mathcal{H}^{\alpha,\rho}}dt\right)\right]<\infty.
    \end{equation*}
\end{enumerate}
Then we can define a new probability measure $Q$ on $\mathcal{F}_t$ by
\begin{equation*}
    \frac{dQ}{dP} := \exp\left(\int_0^{t_1}\langle\theta_t,dW_t\rangle_{\alpha,\rho}-\frac{1}{2}\int_0^{t_1}\Vert \theta_t\Vert^2_{\mathcal{H}^{\alpha,\rho}}dt\right)
\end{equation*}
and under $Q$, the process
\begin{equation*}
    \widetilde{W}_t:=W_t - \int_0^t\theta_sds
\end{equation*}
is a centered Gaussian process over $\mathcal{H}^{\alpha,\rho}$. Formally, for $x\in M$, 
\begin{align*}
\dot{\widetilde{W}}(t,x)=\dot{W}(t,x)-\theta(t,x),
\end{align*}
which remains white in time and colored in space under $Q$. For a deterministic, uniformly bounded $\theta_t$ with $\vert\theta_t\vert<B<\infty$, the integrability and Novikov's condition reduce to verifying
\begin{equation}
\begin{split}
\label{novi}
\int_0^{t_1}\Vert \theta_t\Vert^2_{\mathcal{H}^{\alpha,\rho}}dt&=\int_0^{t_1}\int_{M^2}\theta(t,x)G_{\alpha,\rho}(x,y)\theta(t,y)dxdydt\\
&=\int_0^{t_1}\frac{\rho}{m_0}\left(\int_M\theta(t,x)dx\right)^2 + \sum_{n=1}^\infty\lambda_n^{\alpha}\left(\int_M\theta(t,x)\phi_n(x)dx\right)^2dt\\
&\leq C{t_1}B^2<+\infty.
\end{split}
\end{equation}
We choose $\theta(t,x)=-\frac{P_t(u_0)(x)}{t_1\sigma(t,x)}$. Then equation \eqref{sol} can be rewritten as
\begin{align*}
u(t,x)&=P_t(u_0)(x)+\int_{[0,t]\times M}P_{t-s}(x,y)\sigma(s,y)\left[\widetilde{W}(dsdy)-\frac{P_s(u_0)(y)}{t_1\sigma(s,y)}dsdy\right]\\
&=P_t(u_0)(x)-\frac{tP_t(u_0)(x)}{t_1}+\int_{[0,t]\times M}P_{t-s}(x,y)\sigma(s,y)\widetilde{W}(dsdy)\\
&=\left(1-\frac{t}{t_1}\right)P_t(u_0)(x)+\int_{[0,t]\times M}P_{t-s}(x,y)\sigma(s,y)\widetilde{W}(dsdy).
\end{align*}
The first term vanishes at time $t_1$, and by \eqref{u0convolute},
\begin{equation}
\label{initialbound}
\left|\left(1-\frac{t}{t_1}\right)P_t(u_0)(x)\right|\leq \frac{\mathcal{C}_3}{3}t_1^{h/4}.
\end{equation}

We define
\[\widetilde{N}(t,x):=\int_{[0,t]\times M}P_{t-s}(x,y)\sigma(s,y)\widetilde{W}(dyds).\]
Although $\widetilde{W}$ does not have the same covariance as $W$ under $P$, it is Gaussian under $Q$, and Lemma \ref{larged} applies as
\begin{equation}
    \begin{split}
    \label{Q_bound}
        Q\left(\sup_{\substack{0\leq t\leq c_0\varepsilon^4\\ x\in B_M(p,\sqrt{c_0}\varepsilon^2)}}\vert \widetilde{N}(t,x)\vert>\frac{\mathcal{C}_3}{6}t_1^{h/4}\right)&=Q\left(\sup_{\substack{0\leq t\leq \left(c_0^{1/4}\varepsilon\right)^4\\ x\in B_M(p,\sqrt{c_0}\varepsilon^2)}}\vert \widetilde{N}(t,x)\vert>\frac{\mathcal{C}_3}{6}t_1^{h/4}\right)\\
&\leq C_5\exp\left(-\frac{\mathcal{C}_3^2C_6}{36\mathcal{C}_2^2}\right)<1,
    \end{split}
\end{equation}
where $\gamma=1$ in Lemma \ref{larged}, and the last inequality follows from the definition of $\mathcal C_3$ in \eqref{c3}.

We are now in a position to apply the Gaussian correlation lemma (Lemma \ref{Gaussiancorr}) in order to bound the probability of the event under the measure $Q$. For any $\theta>0$, we define
\begin{equation*}
A_k(\theta):=\left\lbrace\omega:\sup_{(t,x)\in[0,t_1]\times U_k}\vert \widetilde{N}(t,x)(\omega)\vert\leq\theta\right\rbrace,
\end{equation*}
where $\{U_k\}_{k=1}^K$ is a finite cover of $M$. Let $D_k\subset[0,t_1]\times U_k$ be countable and dense, and $D_{k,n}\subset D_k$ is an increasing sequence of finite subsets with $\bigcup\limits_n D_{k,n}=D_k$. Define another sequence of sets
\begin{equation*}
A^{(n)}_k(\theta):=\left\lbrace\omega:\sup_{(t,x)\in D_{k,n}}\vert \widetilde{N}(t,x)(\omega)\vert\leq\theta\right\rbrace.
\end{equation*}
Then $A_k^{(n)}(\theta)$ is a decreasing sequence of events which converges to $A_k(\theta)$ almost surely given that $\widetilde{N}$ has a.s. continuous paths, and therefore
\begin{equation*}
\lim_{n\to\infty}Q\left(\bigcap_{k=1}^KA^{(n)}_k(\theta)\right)=Q\left(\bigcap_{n=1}^\infty\bigcap_{k=1}^KA^{(n)}_k(\theta)\right)=Q\left(\bigcap_{k=1}^K\bigcap_{n=1}^\infty A^{(n)}_k(\theta)\right)= Q\left(\bigcap_{k=1}^KA_k(\theta)\right)
\end{equation*}
and
\begin{equation*}
\lim_{n\to\infty}\prod_{k=1}^KQ\left(A_k^{(n)}(\theta)\right)=\prod_{k=1}^K\lim_{n\to\infty}Q\left(A_k^{(n)}(\theta)\right)=\prod_{k=1}^KQ\left(A_k(\theta)\right),
\end{equation*}
because $K$ is finite and $\prod_{k=1}^Kx_k$ is continuous on $[0,1]^K$. Each $A_k^{(n)}(\theta)$ is a convex, symmetric subset of $\R^{|D_{k,n}|}$, so Lemma \ref{Gaussiancorr} yields
\begin{equation*}
Q\left(\bigcap_{k=1}^KA^{(n)}_{k}(\theta)\right)\geq\prod_{k=1}^KQ\left(A^{(n)}_k(\theta)\right).
\end{equation*}
Taking limits as $n\to\infty$ on both sides gives
\begin{equation}
\label{Gaussiancorrmani}
Q\left(\bigcap_{k=1}^KA_k(\theta)\right)\geq\prod_{k=1}^KQ(A_k(\theta)).
\end{equation}

By \eqref{Dsequence}, \eqref{initialbound}, \eqref{Q_bound}, and \eqref{Gaussiancorrmani},
\begin{equation}
\label{Q}
\begin{split}
Q(D_0)&\geq Q\left(\sup_{\substack{0\leq t\leq t_1\\ x\in M}}\vert \widetilde{N}(t,x)\vert\leq\frac{\mathcal{C}_3}{6}t_1^{h/4}\right)\\
&\geq Q\left(\sup_{\substack{0\leq t\leq t_1\\ x\in B_M(x,\sqrt{c_0}\varepsilon^2)}}\vert \widetilde{N}(t,x)\vert\leq\frac{\mathcal{C}_3}{6}t_1^{h/4}\right)^{\left(\frac{C}{\sqrt{c_0}\varepsilon^2}\right)^d}\\
&\geq \left[1-C_5\exp\left(-\frac{\mathcal{C}_3^2C_6}{36\mathcal{C}_2^2}\right)\right]^{\left(\frac{C}{\sqrt{c_0}\varepsilon^2}\right)^d}.
\end{split}
\end{equation}
Since $\frac{dQ}{dP}$ is a Radon-Nikodym derivative,
\begin{equation}
\label{radon}
1=\E\left[\frac{dQ}{dP}\right]
\end{equation}
and, for a deterministic $\theta_t$, we can estimate the second moment of the Radon-Nikodym derivative by replacing $\theta_t$ with $2\theta_t$,
\begin{equation}
\label{radonnikodym}
\E\left[\left(\frac{dQ}{dP}\right)^2\right]=\exp\left(\int_0^{t_1}\Vert \theta_t\Vert^2_{\mathcal{H}^{\alpha,\rho}}dt\right)\leq \exp\left(Ct_1^{h/2-1}\right)
\end{equation}
where the last inequality comes from \eqref{novi} and \eqref{radon}, and by Cauchy-Schwarz,
\[Q(D_0)\leq \sqrt{\E\left[\left(\frac{dQ}{dP}\right)^2\right]}\cdot\sqrt{P(D_0)}.\]
Combining \eqref{Q} and \eqref{radonnikodym}, we can bound the probability of event $D_0$ under measure $P$ as
\begin{equation}
\label{probD}
\begin{split}
P(D_0)&\geq \exp\left(-Ct_1^{h/2-1}\right)\exp\left(\frac{C'}{c_0^{d/2}\varepsilon^{2d}}\ln\left[1-C_5\exp\left(-\frac{\mathcal{C}_3^2C_6}{36\mathcal{C}_2^2}\right)\right]\right)\\
&\geq \exp\left(-Ct_1^{h/2-1}-C't_1^{-d/2}\right)\\
&\geq C\exp\left(-C't_1^{-d/2}\right),
\end{split}
\end{equation}
since $h=\min(2\alpha-d+2,1)$ and the term $t_1^{-d/2}$ dominates as $\varepsilon\downarrow0$.

\begin{remark}
    This reflects a metric entropy effect: enforcing a uniform smallness event over space–time incurs a cost proportional to the number of essentially independent spatial blocks. The scaling is sharp for sup-norm events of Gaussian fields with spatial regularity $h/2$.
\end{remark}

\subsection*{Non-deterministic $\sigma$}
We now treat the case of random $\sigma(t,x,u)$ via a freezing argument. One can decompose
\[u(t,x)=u_g(t,x)+M(t,x),\]
where $u_g$ solves
\[\partial_t u_g(t,x)=\frac{1}{2}\Delta_M u_g(t,x)+\sigma(t,x,u_0(x))\dot{W}(t,x)\]
and 
\[M(t,x)=\int_{[0,t]\times M}P_{t-s}(x,y)[\sigma(s,y,u(s,y))-\sigma(s,y,u_0(y))]W(dsdy).\]
The process $u_g$ is Gaussian, and under the assumption $\vert u_0(x)\vert\leq \frac{\mathcal{C}_3}{3}t_1^{h/4}$, we define
\[\widetilde{D}_0=\left\lbrace\vert u_g(t_{1},x)\vert\leq \frac{\mathcal{C}_3}{6}t_1^{h/4},\text{and}~\vert u_g(t,x)\vert\leq\frac{2\mathcal{C}_3}{3}t_1^{h/4}~\forall t\in[0,t_{1}],x\in M\right\rbrace.\]
By \eqref{probD}, we acquire
\begin{equation}
\label{probD0}
P\left(\widetilde{D}_0\right)\geq C\exp\left(-C't_1^{-d/2}\right).
\end{equation}

Now we define the stopping time
\[\tau=\inf\left\lbrace t>0:\vert u(t,x)-u_0(x)\vert>t_1^{d/4}\text{~for some $x\in M$}\right\rbrace,\]
with $\tau=+\infty$ if the set is empty, and set
\[\widetilde{M}(t,x)=\int_{[0,t]\times M}P_{t-s}(x,y)[\sigma(s,y,u(s\wedge \tau,y))-\sigma(s,y,u_0(y))]W(dsdy).\]
On $\{\tau>t_1\}$, we have $M(t,x)=\widetilde M(t,x)$ for all $t\le t_1$. Therefore,
\begin{equation}
\label{E0bound}
\begin{split}
P(E_{0})&\geq P\left(\widetilde{D}_0\bigcap \left\lbrace \sup_{\substack{0\leq t\leq t_1\\ x\in M}}\vert M(t,x)\vert\leq\frac{\mathcal{C}_3}{6}t_1^{h/4}\right\rbrace\right)\\
&\geq P\left(\widetilde{D}_0\bigcap \left\lbrace \sup_{\substack{0\leq t\leq t_1\\ x\in M}}\vert M(t,x)\vert\leq\frac{\mathcal{C}_3}{6}t_1^{h/4}\right\rbrace\bigcap\{\tau>t_1\}\right)\\
&\geq P(\widetilde{D}_0)-P\left( \sup_{\substack{0\leq t\leq t_1\\ x\in M}}\vert \widetilde{M}(t,x)\vert>\frac{\mathcal{C}_3}{6}t_1^{h/4}\right)
\end{split}
\end{equation}
with $\vert u(t,x)-u_0(x)\vert\leq t_1^{d/4}$ for all $t\in[0,t_1]$ and $x\in M$. Applying Remark \ref{largeremark} to $\widetilde M$ and using a union bound as in \eqref{probB}, we obtain
\begin{equation}
\label{probDtilde}
P\left( \sup_{\substack{0\leq t\leq t_1\\ x\in M}}\vert \widetilde{M}(t,x)\vert>\frac{\mathcal{C}_3}{6}t_1^{h/4}\right)\leq \frac{C_5}{(c_0\varepsilon^4)^{d/2}}\exp\left(-\frac{C_6}{\mathcal{D}^2 t_1^{d/2}}\right).
\end{equation}
Combining \eqref{probD0}, \eqref{E0bound}, and \eqref{probDtilde}, we conclude that there exists a $\mathcal{D}_0>0$ such that for all $0<\mathcal{D}<\mathcal{D}_0$,
\begin{equation*}
\begin{split}
P(E_{0})&\geq C\exp\left(-C't_1^{-d/2}\right)-\frac{C_5}{t_1^{d/2}}\exp\left(-\frac{C_6}{\mathcal{D}^2 t_1^{d/2}}\right)\\
&\geq \textbf{C}_6\exp\left(-\frac{\textbf{C}_7}{t_1^{d/2}}\right).
\end{split}
\end{equation*}
This establishes \eqref{E0} and completes the proof of Proposition \ref{prop}(b).

\bibliographystyle{alpha}
\bibliography{manifold_small_ball}

\end{document}